\documentclass[11pt]{amsart}

\usepackage{amsmath,amssymb,amsthm}
\usepackage{amsmath,amssymb,amsthm,amscd}
\usepackage[frame,cmtip,arrow,matrix,line,graph,curve]{xy}
\usepackage{graphpap, color}
\usepackage[mathscr]{eucal}
\usepackage{color}

\def\dfrac{\displaystyle\frac}
\def\dsum{\displaystyle\sum}
\def\dmax{\displaystyle\max}

\newtheorem{prop}{Proposition}
\newtheorem{theo}[prop]{Theorem}
\newtheorem{lemm}[prop]{Lemma}
\newtheorem{coro}[prop]{Corollary}
\newtheorem{rema}[prop]{Remark}

\newtheorem{defi}[prop]{Definition}

\newcommand{\al}{\alpha}
\newcommand{\abc}[1]{\left( #1 \right)}
\newcommand{\abz}[1]{\left[ #1 \right]}
\renewcommand{\leq}{\leqslant}
\renewcommand{\geq}{\geqslant}

\newcommand{\p}{\partial}

\numberwithin{equation}{section}

\title{Global $C^2$ estimates for convex solutions of curvature equations}
\author{Pengfei Guan}
\address{Department of Mathematics and Statistics\\ McGill University \\ Montreal, Canada}
\email{guan@math.mcgill.ca}
\author{Changyu Ren}
\address{School of Mathematical Science  \\ Jilin University \\ Changchun, China} \email{rency@jlu.edu.cn}
\author{Zhizhang Wang}
\address{Department of Mathematics and Statistics\\ McGill University \\ Montreal, Canada}
\email{zwang@math.mcgill.ca}
\thanks{Research of the first author was supported in part by an NSERC Discovery Grant. Research of the third author was supported by a CRC postdoctoral fellowship.}
\begin{document}
\begin{abstract}
We establish $C^2$ a priori estimate for convex hypersurfaces whose principal curvatures $\kappa=(\kappa_1,\cdots, \kappa_n)$ satisfying Weingarten curvature equation $\sigma_k(\kappa(X))=f(X,\nu(X))$. We also obtain such estimate for admissible $2$-convex hypersurfaces in the case $k=2$. Our estimates resolve a longstanding problem in geometric fully nonlinear elliptic equations considered in \cite{A2, I1, I, GLL}.
\end{abstract}
\subjclass{53C23, 35J60, 53C42}
\maketitle
\section{introduction}

This paper concerns a longstanding problem of the global $C^2$ estimates for the Weingarten curvature equation in general form
\begin{eqnarray}\label{1.1}
\sigma_k(\kappa(X))=f(X,\nu(X)), \ \ \forall  X\in M,
\end{eqnarray}
where $\sigma_k$ is the $k$th elementary symmetric function, $\nu(X)$ is the outer-normal and $ \kappa(X)=(\kappa_1,\cdots,\kappa_n)$ are principal curvatures of hypersurface $M\subset \mathbb R^{n+1}$ at $X$. The mean curvature, scalar curvature and Gauss curvature correspond to $k=1,2$ and $n$, respectively.

Equation (\ref{1.1}) is associated with many important geometric problems. The Minkowski problem (\cite{N, P1, P3, CY}), the problem of prescribing general Weingarten curvature on outer normals by Alexandrov \cite{A2, gg}, the problem of prescribing curvature measures in convex geometry \cite{A1, P1, GLM, GLL}), the prescribing curvature problem considered \cite{BK, TW, CNS5}, all these geometric problems fall into equation (\ref{1.1}) with special form of $f$ respectively.
Equation (\ref{1.1}) has been studied extensively, it is a special type of general equations systemically studied by Alexandrov in \cite{A2}. When $k=1$, equation (\ref{1.1}) is quasilinear, $C^2$ estimate follows from the classical theory of quasilinear PDE.  The equation is of Monge-Amp\`ere type if $k=n$, $C^2$ estimate in this case for general $f(X, \nu)$ is due to Caffarelli-Nirenberg-Spruck \cite{CNS1}. For the intermediate cases $1<k<n$, $C^2$ estimates have been proved in some special cases.  When $f$ is independent of normal vector $\nu$, $C^2$ estimate has been proved by Caffralli-Nirenberg-Spruck \cite{CNS5} for a general class of fully nonlinear operators $F$, including $F=\sigma_k, F=\frac{\sigma_k}{\sigma_l}$. If $f$ in (\ref{1.1}) depends only on $\nu$, $C^2$ estimate was proved in \cite{gg}. Ivochkina \cite{I1, I} considered the Dirichlet problem of equation (\ref{1.1}) on domains in $\mathbb R^n$, $C^2$ estimate was proved there under some extra conditions on the dependence of $f$ on $\nu$. $C^2$ estimate was also proved for equation of prescribing curvature measures problem in \cite{GLM, GLL}, where $f(X,\nu)= \langle X,\nu \rangle \tilde f(X)$.

\medskip

We establish $C^2$ estimate for convex solutions of equation (\ref{1.1}) for $1<k<n$ and $C^2$ estimate for admissible solutions of equation (\ref{1.1}) for $k=2$.
$C^2$ estimates for equation (\ref{1.1}) is equivalent to the curvature estimates from above for $\kappa_1, \cdots, \kappa_n$. We state the main results of this paper.

\begin{theo}\label{theo 11}
Suppose $M\subset \mathbb R^{n+1}$ is a closed convex hypersurface satisfying curvature equation (\ref{1.1}) for some positive function $f(X, \nu)\in C^{2}(\Gamma)$, where $\Gamma$ is an open neighborhood of unit normal bundle of $M$ in $\mathbb R^{n+1} \times \mathbb S^n$, suppose $M$ encloses a ball of radius $r_0>0$, then there is a constant $C$ depending only on $n, k$, $r_0$, $\inf f$ and $\|f\|_{C^2}$, such that
 \begin{equation}\label{Mc2}
 \max_{X\in M, i=1,\cdots, n} \kappa_i(X) \le C.\end{equation}
\end{theo}

Estimate (\ref{Mc2}) is special to equation (\ref{1.1}). One may ask if estimate (\ref{Mc2}) can be generalized to this type of curvature equations when $f$ depends on $(X,\nu)$ as in (\ref{1.1}). The answer is {\it no} in general.

\begin{theo}\label{theo 11-counter}
For each $1\le l< k\le n$, there exist $C>0, r_0>0$ and a sequence of smooth positive functions $f_t(X,\nu)$ with
\[\|f_t\|_{C^3(\mathbb R^{n+1}\times \mathbb S^n)}+\|\frac1{f_t}\|_{C^3(\mathbb R^{n+1}\times \mathbb S^n)}\le C,\]
and a sequence of strictly convex hypersurface $M_t\subset \mathbb R^{n+1}$ with  $M_t$ encloses $B_{r_0}(0)$ satisfying quotient of curvatures equation
\begin{equation}\label{cur-qq}
\frac{\sigma_{k}}{\sigma_{l}}(\kappa)=f_t(X,\nu),\end{equation}
such that estimate (\ref{Mc2}) fails.
\end{theo}

\medskip

It is desirable to drop the convexity assumption in Theorem \ref{theo 11}. In the case of scalar curvature equation ($k=2$), we establish estimate (\ref{Mc2}) for starshaped admissible solutions of equation (\ref{1.1}). The general case $2<k<n$ is still open.

\medskip

Following \cite{CNS3}, we define
\begin{defi}\label{k-convex} For a domain $\Omega\subset \mathbb R^n$, a function $v\in C^2(\Omega)$ is called $k$-convex if the eigenvalues $\lambda (x)=(\lambda_1(x), \cdots, \lambda_n(x))$ of the hessian $\nabla^2 v(x)$ is in $\Gamma_k$ for all $x\in \Omega$, where $\Gamma_k$ is the Garding's cone
\[\Gamma_k=\{\lambda \in \mathbb R^n \ | \quad \sigma_m(\lambda)>0, \quad  m=1,\cdots,k\}.\]

A $C^2$ regular hypersurface $M\subset \mathbb R^{n+1}$ is
$k$-convex if $\kappa(X)\in \Gamma_k$
for all $X\in M$.
 \end{defi}

\begin{theo}\label{theo2 11}
Suppose $k=2$ and suppose $M\subset \mathbb R^{n+1}$ is a closed starshaped $2$-convex hypersurface satisfying curvature equation (\ref{1.1}) for some positive function $f(X, \nu)\in C^{2}(\Gamma)$, where $\Gamma$ is an open neighborhood of unit normal bundle of $M$ in $\mathbb R^{n+1} \times \mathbb S^n$, suppose $M=\{\rho(x)x | x\in \mathbb S^n\}$ is starshaped with respect to the origin, then there is a constant $C$ depending only on $n, k$, $\inf \rho$, $\|\rho\|_{C^1}$, $\inf f$ and $\|f\|_{C^2}$, such that
 \begin{equation}\label{M2c2}
 \max_{X\in M,  i=1,\cdots, n} \kappa_i(X)  \le C.\end{equation}
\end{theo}

Theorem \ref{theo 11} and Theorem \ref{theo2 11} are stated for compact hypersurfaces, the corresponding estimates hold for solutions of equation (\ref{1.1}) with boundary conditions, with $C$ in the right hand side of (\ref{Mc2}) and (\ref{M2c2}) depending $C^2$ norm on the boundary in addition.

The proof of above two theorems relies on maximum principles for appropriate curvature functions. The novelty of this paper is the discovery of some new test curvature functions. They are nonlinear in terms of the principal curvatures with some good convexity properties.

\medskip

With appropriate barrier conditions on function $f$, one may establish existence results of the prescribing curvature problem (\ref{1.1}) in general.

\begin{theo}\label{k-exist} Suppose $f\in C^2(\mathbb R^{n+1}\times \mathbb S^n)$ is a positive function and suppose there is a constant $r>1$ such that,
\begin{eqnarray}\label{4.101}
f(X,\frac{X}{|X|})\leq \frac{\sigma_k(1,\cdots, 1)}{r^k}\ \ \text{ for }\ \   |X|=r ,
\end{eqnarray}
 and $f^{-1/k}(X,\nu)$ is a locally convex in $X \in B_r(0)$  for any fixed $\nu\in \mathbb S^n$, then equation (\ref{1.1}) has a strictly convex $C^{3,\alpha}$ solution inside $\bar B_r$.
\end{theo}

To state a corresponding existence result for $2$-convex solutions of the prescribed scalar curvature equation (\ref{1.1}), we need further barrier conditions on the prescribed function $f$ as considered in \cite{BK, TW, CNS5}. We denote $\rho(X)=|X|$.

\medskip

We assume that

\noindent {\it Condition} (1). There are two positive constant $r_1<1<r_2$ such that
\begin{equation}\label{4.1}
\left\{
\begin{matrix}
f(X,\frac{X}{|X|}) &\geq&  \frac{\sigma_k(1,\cdots, 1)}{r^k_1},\ \ \text{ for } |X|=r_1,\\
f(X,\frac{X}{|X|}) &\leq&  \frac{\sigma_k(1,\cdots, 1)}{r_2^k}, \ \ \text{  for  } |X|=r_2 .
\end{matrix}\right.
\end{equation}

\noindent {\it Condition} (2). For any fixed unit vector $\nu$,
\begin{eqnarray}\label{4.2}
\frac{\p }{\p \rho}(\rho^kf(X,\nu))\leq 0,\ \   \text{ where } |X|=\rho.
\end{eqnarray}

\begin{theo}\label{2-exist} Suppose $k=2$ and suppose positive function $f\in C^2(\bar B_{r_2}\setminus B_{r_1}\times \mathbb S^n)$ satisfies conditions (\ref{4.1}) and (\ref{4.2}), then equation (\ref{1.1}) has a unique $C^{3,\alpha}$ starshaped solution $M$ in $\{r_1\le |X|\le r_2\}$.
\end{theo}

\bigskip

The organization of the paper is as follow. As an illustration, we give a short proof of $C^2$ estimate for $\sigma_2$-Hessian equation on $\mathbb R^2$ in Section 2. Theorem \ref{theo2 11} and Theorem \ref{theo 11} are proved in Section 3 and Section 4 respectively. Section 5 is devoted to various existence theorems. Construction of examples of convex hypersurfaces stated in Theorem \ref{theo 11-counter} appears in Section 6.

\section{The Hessian equation for $k=2$}
\par

We first consider $\sigma_2$-Hessian equations in a domain $\Omega \subset \mathbb R^{n+1}$:
\begin{equation}\label{2.1}
\left\{\begin{matrix}\sigma_2[D^2u]&=&f(x,u,Du), \\
 u|_{\partial \Omega}&=&\phi.\end{matrix}\right.
\end{equation}

We believe $C^2$ estimates for equation (\ref{2.1}) is known. Since we are not able to find any reference in the literature, a proof is produced here to serve as an illustration.

We need following lemma which is a slightly improvement of Lemma 1 in \cite{GLL}.

\begin{lemm}\label{lemma 1}
Assume that $k>l$, $W=(w_{ij})$ is a Codazzi tensor which is in $\Gamma_k$. Denote $\al=\dfrac{1}{k-l}$.  Then, for $h=1,\cdots, n$, we have the following inequality,
\begin{eqnarray}\label{1.4}
&&-\dfrac{\sigma_k^{pp,qq}}{\sigma_k}(W)w_{pph}w_{qqh}+\dfrac{\sigma_l^{pp,qq}}{\sigma_l}(W)
w_{pph}w_{qqh}\\
&\geq& \abc{\dfrac{(\sigma_k(W))_h}{\sigma_k(W)}-\dfrac{(\sigma_l(W))_h}{\sigma_l(W)}}
\abc{(\al-1)\dfrac{(\sigma_k(W))_h}{\sigma_k(W)}-(\al+1)\dfrac{(\sigma_l(W))_h}{\sigma_l(W)}}.\nonumber
\end{eqnarray}
Furthermore, for any $\delta>0$,
\begin{eqnarray}\label{1.5}
&&-\sigma_k^{pp,qq}(W)w_{pph}w_{qqh} +(1-\al+\dfrac{\al}{\delta})\dfrac{(\sigma_k(W))_h^2}{\sigma_k(W)}\\
&\geq &\sigma_k(W)(\al+1-\delta\al)
\abz{\dfrac{(\sigma_l(W))_h}{\sigma_l(W)}}^2
-\dfrac{\sigma_k}{\sigma_l}(W)\sigma_l^{pp,qq}(W)w_{pph}w_{qqh}.\nonumber
\end{eqnarray}
\end{lemm}
\begin{proof} Define a function $$\ln F=\ln (\dfrac{\sigma_k}{\sigma_l})^{1/(k-l)}=\dfrac{1}{k-l}\ln\sigma_k-\dfrac{1}{k-l}\ln\sigma_l.$$ Differentiate it twice,
$$
\dfrac{F^{pp}}{F}=\dfrac{1}{k-l}\dfrac{\sigma_k^{pp}}{\sigma_k}-\dfrac{1}{k-l}\dfrac{\sigma_l^{pp}}{\sigma_l},
$$
$$
\dfrac{F^{pp,qq}}{F}-\dfrac{F^{pp}F^{qq}}{F^2}=\dfrac{1}{k-l}\dfrac{\sigma_k^{pp,qq}}{\sigma_k}
-\dfrac{1}{k-l}\dfrac{\sigma_k^{pp}\sigma_k^{qq}}{\sigma_k^2}
-\dfrac{1}{k-l}\dfrac{\sigma_l^{pp,qq}}{\sigma_l}+\dfrac{1}{k-l}\dfrac{\sigma_l^{pp}\sigma_l^{qq}}{\sigma_l^2}.
$$
By the concavity of $F$ and using previous two equalities,
$$
-\dfrac{1}{k-l}\abc{\dfrac{\sigma_k^{pp}}{\sigma_k}-\dfrac{\sigma_l^{pp}}{\sigma_l}}\abc{\dfrac{\sigma_k^{qq}}{\sigma_k}
-\dfrac{\sigma_l^{qq}}{\sigma_l}}
\geq \dfrac{\sigma_k^{pp,qq}}{\sigma_k}-\dfrac{\sigma_k^{pp}\sigma_k^{qq}}{\sigma_k^2}
-\dfrac{\sigma_l^{pp,qq}}{\sigma_l}+\dfrac{\sigma_l^{pp}\sigma_l^{qq}}{\sigma_l^2}.
$$
Here the meaning of  $"\geq"$ is for comparison of positive definite matrices. Hence, for each $h$ with $(w_{11h},\cdots, w_{nnh})$, we obtain
(\ref{1.4}). (\ref{1.5}) follows from (\ref{1.4}) and the Schwarz inequality.
\end{proof}

\bigskip

Consider
$$
\phi=\dmax_{|\xi|=1,x\in\Omega}\exp\{\dfrac{\varepsilon}{2}|Du|^2+\dfrac{a}{2}|x|^2\}u_{\xi\xi},
$$
where $\varepsilon$ to be determined later.
Suppose that the maximum of $\phi$ is achieved  at some point $x_0$ in $\Omega$ along
some direction $\xi$. We may assume that $\xi=(1,0,\cdots,0)$.
Rotating the coordinates if necessary, we may assume the matrix $(u_{ij})$ is diagonal, and  $u_{11}\geq u_{22}\cdots\geq u_{nn}$ at the point.
\par
Differentiate the function twice  at  $x_0$,
\begin{equation}\label{2.3}
\dfrac{u_{11i}}{u_{11}}+\varepsilon u_iu_{ii}+ax_i=0,
\end{equation}
and
\begin{equation}\label{2.4}
\dfrac{u_{11ii}}{u_{11}}-\dfrac{u_{11i}^2}{u_{11}^2}+\dsum_k\varepsilon
u_ku_{kii}+\varepsilon u_{ii}^2+a\leq 0.
\end{equation}
Contract with the matrix
$\sigma_2^{ii}u_{11}$,
\begin{equation}\label{2.5}
\sigma_2^{ii}u_{11ii}-\sigma_2^{ii}\dfrac{u_{11i}^2}{u_{11}}+u_{11}\dsum_k\varepsilon
u_k\sigma_2^{ii}u_{kii}+u_{11}\varepsilon \sigma_2^{ii}u_{ii}^2+a\sum_i\sigma_2^{ii}u_{11}\leq
0.
\end{equation}
\par
At $x_0$, differentiate equation (\ref{1.1}) twice,
\begin{eqnarray}\label{2.6}
\sigma_2^{ii}u_{iij}&=&f_j+f_uu_j+f_{p_j}u_{jj},
\end{eqnarray}
and
\begin{eqnarray}\label{2.7}
&&\sigma_2^{ii}u_{iijj}+\sigma_2^{pq,rs}u_{pqj}u_{rsj}\\
&=&f_{jj}+2f_{ju}u_j+2f_{jp_j}u_{jj}+f_{uu}u_j^2+2f_{up_j}u_ju_{jj}+f_uu_{jj}+f_{p_jp_j}u_{jj}^2+\sum_kf_{p_k}u_{kjj}\nonumber.
\end{eqnarray}
Choose $j=1$ in the above equation, and insert (\ref{2.7}) into (\ref{2.5}),
\begin{align*}
0\geq &-C-Cu_{11} +f_{p_1p_1}u_{11}^2+\dsum_kf_{p_k}u_{k11} -\sigma_2^{pq,rs}u_{pq1}u_{rs1}\\
&-\sigma_2^{ii}\frac{u_{11i}^2}{u_{11}}+u_{11}\dsum_k\varepsilon
u_k\sigma_2^{ii}u_{kii}+u_{11}\varepsilon \sigma_2^{ii}u_{ii}^2+a\sum_i\sigma_2^{ii}u_{11}.
\end{align*}

Use (\ref{2.3}) and (\ref{2.6}),
$$
\dsum_kf_{p_k}u_{k11}+u_{11}\dsum_k\varepsilon
u_k\sigma_2^{ii}u_{kii}=u_{11}\dsum_k(\varepsilon u_k f_k+\varepsilon
f_uu_k^2-ax_kf_{p_k}).
$$
Then
\begin{eqnarray}\label{2.8}
0&\geq& -C-Cu_{11} +f_{p_1p_1}u_{11}^2
-\dsum_{p\neq r}u_{pp1}u_{rr1}
+\dsum_{p\neq
q}u_{pq1}^2-\sigma_2^{ii}\dfrac{u_{11i}^2}{u_{11}}\\
&&+u_{11}\varepsilon
\sigma_2^{ii}u_{ii}^2+(n-1)au_{11}\dsum_ku_{kk}.\nonumber
\end{eqnarray}
Choose $k=2,l=1$ and $h=1$ in Lemma \ref{lemma 1}, we have,
$$
-\dsum_{p\neq r}u_{pp1}u_{rr1}+(1-\alpha+\frac{\alpha}{\delta})\frac{(\sigma_2)_1^2}{\sigma_2}
\ \ \geq \ \ (\alpha +1-\delta\alpha )\sigma_2[\frac{(\sigma_1)_1}{\sigma_1}]^2\ \ \geq \ \ 0.
$$
Inequality (\ref{2.8}) becomes,
\begin{eqnarray}\label{2.9}
0&\geq& -C-Cu_{11} +f_{p_1p_1}u_{11}^2 +(n-1)au_{11}^2-C(\sigma_2)_1^2\\
&&+u_{11}\varepsilon \sigma_2^{ii}u_{ii}^2+2\dsum_{k\neq
1}u_{11k}^2-\sigma_2^{ii}\frac{u_{11i}^2}{u_{11}}\nonumber\\
&\geq&((n-1)a-C_0)u_{11}^2+u_{11}\varepsilon \sigma_2^{ii}u_{ii}^2+2\dsum_{k\neq
1}u_{11k}^2-\sigma_2^{ii}\frac{u_{11i}^2}{u_{11}}\nonumber,
\end{eqnarray}
where we have used (\ref{2.3}) and the Schwarz inequality. We claim, if $a$ is chosen sufficient large such that
$$
(n-1)a-C_0 \geq 1,
$$
then
\begin{equation}\label{2.0}
u_{11}\varepsilon \sigma_2^{ii}u_{ii}^2+2\dsum_{k\neq
1}u_{11k}^2-\sigma_2^{ii}\frac{u_{11i}^2}{u_{11}} \geq 0.
\end{equation}
(\ref{2.9}) will yield an upper bound of $u_{11}$.
\par
We prove the claim (\ref{2.0}). We may
assume that $u_{11}$ is sufficient large. By (\ref{2.3}) and the Schwarz inequality,
\begin{eqnarray}\label{2.10}
\sigma_2^{11}u_{11}\varepsilon u_{11}^2-\sigma_2^{11}\dfrac{u_{111}^2}{u_{11}}
&\geq&\sigma_2^{11}u_{11}(\varepsilon
u_{11}^2-2\varepsilon^2u_1^2u_{11}^2-2a^2x_1^2).
\end{eqnarray}
If we require
\begin{eqnarray}\label{2.11}
\varepsilon \geq 3\varepsilon^2\max_{\Omega}|\nabla u|^2,
\end{eqnarray}
and if $u_{11}$ sufficient large, (\ref{2.10}) is nonnegative.  As in \cite{CW}, we divide it into two different cases.  Denote $\lambda_i=u_{ii}$.

\par
(A)\ \  $\dsum_{i=2}^{n-1}\lambda_i\leq \lambda_1$. In this case, for $i\neq 1$, since $\lambda_1\geq \lambda_2\geq\cdots\geq \lambda_n$,
$$2u_{11}\geq \sigma_2^{ii}.$$ Hence,
$$2\dsum_{k\neq 1}u_{11k}^2-\sum_{i\neq 1}\sigma_2^{ii}\dfrac{u_{11i}^2}{u_{11}}\geq 0.$$
Combine with (\ref{2.10}), we obtain (\ref{2.0}).
\par
(B) \ \  $\dsum_{i=2}^{n-1}\lambda_i\geq \lambda_1$, then
$\dfrac{\lambda_1}{n-2}\leq\lambda_2\leq\lambda_1$. We further divide
this case into two subcases.
\par
(B1)\ \  Suppose  $\sigma_2^{22}\geq 1$. Using (\ref{2.11}), (\ref{2.3}) and the Schwarz inequality,
\begin{align*}
&u_{11}\varepsilon \sum_{i\neq 1}\sigma_2^{ii}u_{ii}^2-\sum_{i\neq 1}\sigma_2^{ii}\frac{u_{11i}^2}{u_{11}}\\
= &\sigma_2^{22}u_{11}\varepsilon u_{22}^2-\sigma_2^{22}\dfrac{u_{112}^2}{u_{11}}+\dsum_{i>2}(\sigma_2^{ii}u_{11}\varepsilon u_{ii}^2-\sigma_2^{ii}\dfrac{u_{11i}^2}{u_{11}})\\
\geq&\sigma_2^{22}u_{11}(\varepsilon u_{22}^2-2\varepsilon^2 u_2^2u_{22}^2-2a^2x_2^2)+\dsum_{i>2}\sigma_2^{ii}u_{11}(\varepsilon u_{ii}^2-2\varepsilon^2 u_i^2u_{ii}^2 -2a^2x_i^2)\\
\geq& \dfrac{1}{3}\sigma_2^{22}u_{11}(\varepsilon u_{22}^2-C)-Cu_{11}\sum_{i>2}\sigma_2^{ii}\\
\geq& \dfrac{\varepsilon}{3}\lambda_1 \lambda_2^2-C\lambda_1^2\\
\geq& \dfrac{\varepsilon}{3(n-2)^2}\lambda^3_1-C\lambda_1^2,
\end{align*}
it is nonnegative if $\lambda_1$ is sufficient large. In view of (\ref{2.10}), in this subcase, (\ref{2.0}) holds.
\par
(B2)\ \  Suppose $\sigma_2^{22}< 1$. Again, we may assume that $\lambda_1$ is sufficient large, we have $\lambda_n<0$.  By the assumption,
$1\geq \lambda_1+(n-2)\lambda_n$. It implies,
$$-\lambda_n\geq\frac{\lambda_1-1}{n-2}.$$  Since $\sigma_2^{nn}+\lambda_n= \lambda_1+\sigma_2^{11}$, we have $\sigma^{nn}_2\geq \lambda_1$. We get
\begin{align*}
&u_{11}\varepsilon \sum_{i\neq 1}\sigma_2^{ii}u_{ii}^2-\sum_{i\neq 1}\sigma_2^{ii}\frac{u_{11i}^2}{u_{11}}\\
= &\sigma_2^{nn}u_{11}\varepsilon u_{nn}^2-\sigma_2^{nn}\dfrac{u_{11n}^2}{u_{nn}}+\sum_{1< i<n}(\sigma_2^{ii}u_{11}\varepsilon u_{ii}^2-\sigma_2^{ii}\dfrac{u_{11i}^2}{u_{11}})\\
\geq& \dfrac{1}{3}\sigma_2^{nn}u_{11}(\varepsilon u_{nn}^2-C)-Cu_{11}\sum_{1<i<n}\sigma_2^{ii}\\
\geq& \dfrac{\varepsilon}{3(n-2)^2}\lambda_1^2(\lambda_1-1)^2-C\lambda_1^2.
\end{align*}
Here, the first inequality comes from (\ref{2.3}) and the Schwarz inequality. The process is similar to the first and second inequalities in subcase (B1). The above quantity is nonnegative, if $\lambda_1$ is sufficient large. (\ref{2.0}) follows from (\ref{2.10}).

\medskip

With the $C^2$ interior estimate, one may obtain a global $C^2$ estimate if the corresponding boundary estimate is in hand. This type of $C^2$ boundary estimates have been proved by Bo Guan in \cite{B} under the assumption that Dirichlet problem (\ref{2.1}) has a subsolution. Namely, there is a function $\underline{u}$, satisfying
 \begin{equation}\label{2.1.1}
\left\{\begin{matrix}\sigma_2[D^2\underline{u}]&\geq &f(x,\underline{u},D\underline{u}), \\
\underline{u}|_{\partial \Omega}&=&\phi.\end{matrix}\right.
\end{equation}

\begin{theo}
Suppose $\Omega\subset \mathbb R^n$ is a bounded domain with smooth boundary. Suppose $f(p, u, x)\in C^2(\mathbb R^n\times \mathbb R\times \bar\Omega)$ is a positive function with $f_u\ge 0$. Suppose there is a subsolution $\underline{u}\in C^3(\bar \Omega)$ satisfying (\ref{2.1.1}), then the Dirichlet problem (\ref{2.1}) has a unique $C^{3,\alpha}, \forall 0<\alpha<1$ solution $u$. \end{theo}

\medskip

To conclude this section, we list one lemma which is well known (e.g., Theorem 5.5 in \cite{Ball}, it was also originally stated in a preliminary version of \cite{CNS3} and was lately removed from the published version).

\begin{lemm} Denote $Sym(n)$ the set of all $n\times n$ symmetric matrices. Let $F$ be a $C^2$ symmetric function defined in some open subset $\Psi \subset Sym(n)$. At any diagonal matrix $A\in \Psi$ with distinct
eigenvalues, let $\ddot{F}(B,B)$ be the second derivative of $C^2$ symmetric function $F$
in direction $B \in Sym(n)$, then
\begin{eqnarray}
\label{1.0} \ddot{F}(B,B) =  \sum_{j,k=1}^n {\ddot{f}}^{jk}
B_{jj}B_{kk} + 2 \sum_{j < k} \frac{\dot{f}^j -
\dot{f}^k}{{\lambda}_j - {\lambda}_k} B_{jk}^2.
\end{eqnarray}
\end{lemm}

\section{the scalar curvature equation}

We consider the global curvature estimates for solution to curvature equation (\ref{1.1}) with $k=2$, i.e. the prescribing scalar curvature equation in $\mathbb R^{n+1}$. In \cite{Gerh}, a global curvature estimate was obtained for prescribing scalar curvature equation in Lorentzian manifolds, where some special properties of the spacelike hypersurfaces were used. It seems for equation (\ref{1.1}) in $\mathbb R^{n1+}$, the situation is different. A new feature here is to consider a nonlinear test function $\log \sum_l e^{\kappa_l}$. We explore certain convexity property of this function, which will be used in a crucial way in our proof.

Set $u(X)=<X, \nu(X)>$. By the assumption that $M$ is starshaped with a $C^1$ bound, $u$ is bounded from below and above by two positive constants.
At every point in the hypersurface $M$, choose a local coordinate frame $\{ \p/(\p x_1),\cdots,\p/(\p x_{n+1})\}$ in $\mathbb{R}^n$ such that the first $n$ vectors are the local coordinates of the hypersurface and the last one is the unit outer normal vector.  Denote $\nu$ to be the outer normal vector. We let  $h_{ij}$ and $u$ be the second fundamental form and the support function of the hypersurface $M$ respectively.  The following geometric formulas are well known (e.g., \cite{GLL}).

\begin{equation}
h_{ij}=\langle\partial_iX,\partial_j\nu\rangle,
\end{equation}
and
\begin{equation}
\begin{array}{rll}
X_{ij}=& -h_{ij}\nu\quad {\rm (Gauss\ formula)}\\
(\nu)_i=&h_{ij}\partial_j\quad {\rm (Weigarten\ equation)}\\
h_{ijk}=& h_{ikj}\quad {\rm (Codazzi\ formula)}\\
R_{ijkl}=&h_{ik}h_{jl}-h_{il}h_{jk}\quad {\rm (Gauss\ equation)},\\
\end{array}
\end{equation}
where $R_{ijkl}$ is the $(4,0)$-Riemannian curvature tensor. We also have
\begin{equation}
\begin{array}{rll}
h_{ijkl}=& h_{ijlk}+h_{mj}R_{imlk}+h_{im}R_{jmlk}\\
=& h_{klij}+(h_{mj}h_{il}-h_{ml}h_{ij})h_{mk}+(h_{mj}h_{kl}-h_{ml}h_{kj})h_{mi}.\\
\end{array}
\end{equation}

\medskip

We need a more explicit version of Lemma \ref{lemma 1}.
\begin{lemm}\label{exch}
Suppose $W=(w_{ij})$ is a Codazzi tensor which is in $\Gamma_2$.  For $h=1,\cdots, n$, there exist sufficient large constants $K,\alpha$ and sufficient small constant $\delta$,  such that the following inequality holds,
 \begin{eqnarray}\label{e3.4}
 K(\sigma_2)_h^2-\sum_{p\neq r}w_{pph}w_{rrh}-\delta w_{hh}\sigma_2^{hh}\frac{w_{hhh}^2}{\sigma_1^2}+\alpha \sum_{i\neq h}w_{iih}^2\geq 0.
 \end{eqnarray}
 \end{lemm}
\begin{proof}
Consider function $$Q=\frac{\sigma_2(W)}{\sigma_1(W)}.$$ We have,
$$\sigma_1Q^{pp,qq}w_{pph}w_{qqh}=\sum_{p\neq q}w_{pph}w_{qqh}-\frac{2(\sigma_2)_h\sum_jw_{jjh}}{\sigma_1}+2\frac{\sigma_2(\sum_{j}w_{jjh})^2}{\sigma_1^2}.$$ On the other hand, one may write (e.g. \cite{HS})
$$-Q^{pp,qq}w_{pph}w_{qqh}=\frac{\sum_i(w_{iih}\sigma_1-w_{ii}\sum_kw_{kkh})^2}{\sigma_1^3}.$$ From the above two identities and the Schwartz inequality, with $K, \alpha$ large enough,
\begin{eqnarray}\label{e3.5}
&&-\sum_{p\neq r}w_{pph}w_{rrh}\\
&=&\frac{\sum_i(w_{iih}\sigma_1-w_{ii}\sum_kw_{kkh})^2}{\sigma_1^2}-\frac{2(\sigma_2)_h\sum_jw_{jjh}}{\sigma_1}+2\frac{\sigma_2(\sum_{j}w_{jjh})^2}{\sigma_1^2}\nonumber\\
&\geq&\frac{\sigma_2(\sum_{j}w_{jjh})^2}{\sigma_1^2}-K (\sigma_2)_h^2+\frac{(w_{hhh}\sigma_1-w_{hhh}w_{hh}-w_{hh}\sum_{k\neq h}w_{kkh})^2}{\sigma_1^2}\nonumber\\
&&+\frac{\sum_{i \neq h}(w_{iih}\sigma_1-w_{ii}w_{hhh}-w_{ii}\sum_{k\neq h}w_{kkh})^2}{\sigma_1^2}\nonumber\\
&\geq&\frac{\sigma_2(w_{hhh})^2}{\sigma_1^2}-K (\sigma_2)_h^2+\frac{(w_{hhh}\sigma_2^{hh})^2}{2\sigma_1^2}+\frac{w_{hhh}^2\sum_{i\neq h}w_{ii}^2}{2\sigma_1^2}-\alpha\sum_{i\neq h}w_{iih}^2.\nonumber
\end{eqnarray}
By \eqref{e3.5},
\begin{eqnarray}
 &&K(\sigma_2)_h^2-\sum_{p\neq r}w_{pph}w_{rrh}-\delta\sigma_2^{hh}w_{hh}\frac{w_{hhh}^2}{\sigma_1^2}+\alpha \sum_{i\neq h}w_{iih}^2\nonumber\\
 &\geq&\frac{\sigma_2(w_{hhh})^2}{\sigma_1^2}+\frac{w_{hhh}^2\sum_{i\neq h}w_{ii}^2}{2\sigma_1^2}-\delta\frac{\sigma_2^{hh}w_{hh}w_{hhh}^2}{\sigma_1^2}\nonumber.
  \end{eqnarray}
  Since, $$w_{hh}\sigma_2^{hh}=\sigma_2-\frac{1}{2}\sum_{a\neq b;a,b\neq h}w_{aa}w_{bb},$$ if $\delta$ is sufficient small, we obtain \eqref{e3.4}.
 \end{proof}

\bigskip

Theorem \ref{theo2 11} is a consequence of the following theorem.
\begin{theo}\label{theo2 11-0}
Suppose $k=2$ and suppose $M\subset \mathbb R^{n+1}$ is a starshaped $2$-convex hypersurface satisfying curvature equation (\ref{1.1}) for some positive function $f(X, \nu)\in C^{2}(\Gamma)$, where $\Gamma$ is an open neighborhood of unit normal bundle of $M$ in $\mathbb R^{n+1} \times \mathbb S^n$, then there is a constant $C$ depending only on $n, k$, $\|M\|_{C^1}$, $\inf f$ and $\|f\|_{C^2}$, such that
 \begin{equation}\label{M2c2-00}
 \max_{X\in M, i=1,\cdots, n} \kappa_i(X)  \le C(1+ \max_{X\in \partial M, i=1,\cdots, n} \kappa_i(X)).\end{equation}
\end{theo}

Set
\begin{eqnarray}
P=\sum_le^{\kappa_l}, \quad
\phi=\log\log P-(1+\varepsilon)\log u+\frac{a}{2}|X|^2,\end{eqnarray}
where $\varepsilon$ and $a$ are constants which will be determined later.  We may assume that the maximum of
$\phi$ is achieved  at some point $X_0\in M$. After
rotating the coordinates, we may assume the matrix $(h_{ij})$ is diagonal at the point, and we can further  assume that $h_{11}\geq h_{22}\cdots\geq h_{nn}$. Denote $\kappa_i=h_{ii}$.
\par
Differentiate the function twice  at  $X_0$,
\begin{equation}\label{e2.3}
\phi_i=\dfrac{P_i}{P\log P}- (1+\varepsilon)\frac{h_{ii}\langle X,\p_i\rangle}{u}+a\langle \p_i, X\rangle=0,
\end{equation}
and by (\ref{1.0}),
\begin{eqnarray}\label{e2.4}
&&\phi_{ii}\\
&=& \frac{P_{ii}}{P\log P}-\frac{P_i^2}{P^2\log P}-\frac{P_i^2}{(P\log P)^2}- \frac{1+\varepsilon}{u}\sum_lh_{il,i}\langle \p_l,X \rangle-\frac{(1+\varepsilon) h_{ii}}{u}\nonumber\\&&+(1+\varepsilon)h_{ii}^2+(1+\varepsilon)\frac{h_{ii}^2\langle X,\p_i\rangle^2}{u^2}+a-aUh_{ii}\nonumber\\
&=&\frac{1}{P\log P}[\sum_le^{\kappa_l}h_{llii}+\sum_le^{\kappa_l}h_{lli}^2+\sum_{\alpha\neq \beta}\frac{e^{\kappa_{\alpha}}-e^{\kappa_{\beta}}}{\kappa_{\alpha}-\kappa_{\beta}}h_{\alpha\beta i}^2-(\frac{1}{P}+\frac{1}{P\log P})P_i^2]\nonumber\\
&&- \frac{(1+\varepsilon)\sum_lh_{iil}\langle \p_l,X \rangle}{u}-\frac{ (1+\varepsilon)h_{ii}}{u}+(1+\varepsilon)h_{ii}^2+(1+\varepsilon)\frac{h_{ii}^2\langle X,\p_i\rangle^2}{u^2}\nonumber\\&&+a-aUh_{ii}\nonumber\\
&=&\frac{1}{P\log P}[\sum_le^{\kappa_l}h_{ii,ll}+\sum_le^{\kappa_l}(h_{il}^2-h_{ii}h_{ll})h_{ii}+\sum_le^{\kappa_l}(h_{ii}h_{ll}-h_{il}^2)h_{ll}\nonumber\\
&&+\sum_le^{\kappa_l}h_{lli}^2+\sum_{\alpha\neq \beta}\frac{e^{\kappa_{\alpha}}-e^{\kappa_{\beta}}}{\kappa_{\alpha}-\kappa_{\beta}}h_{\alpha\beta i}^2-(\frac{1}{P}+\frac{1}{P\log P})P_i^2]\nonumber\\
&&- \frac{(1+\varepsilon)\sum_lh_{iil}\langle \p_l,X \rangle}{u}-\frac{ (1+\varepsilon)h_{ii}}{u}+(1+\varepsilon)h_{ii}^2+(1+\varepsilon)\frac{h_{ii}^2\langle X,\p_i\rangle^2}{u^2}\nonumber\\&&+a-aUh_{ii}\nonumber
\end{eqnarray}
Contract with $\sigma_2^{ii}$,
\begin{eqnarray}\label{e2.5}
&&\sigma_2^{ii}\phi_{ii}\\
&=&\frac{1}{P\log P}[\sum_le^{\kappa_l}\sigma_2^{ii}h_{ii,ll}+2f\sum_le^{\kappa_l}h_{ll}^2-\sigma_2^{ii}h_{ii}^2\sum_le^{\kappa_l}h_{ll}+\sum_l\sigma_2^{ii}e^{\kappa_l}h_{lli}^2\nonumber\\&&+\sum_{\alpha\neq \beta}\sigma_2^{ii}\frac{e^{\kappa_{\alpha}}-e^{\kappa_{\beta}}}{\kappa_{\alpha}-\kappa_{\beta}}h_{\alpha\beta i}^2-(\frac{1}{P}+\frac{1}{P\log P})\sigma_2^{ii}P_i^2]+(n-1)a\sigma_1-2afuU\nonumber\\
&&- \frac{(1+\varepsilon)\sum_l\sigma_2^{ii}h_{iil}\langle \p_l,X \rangle}{u}-\frac{ (1+\varepsilon)2f}{u}+(1+\varepsilon)\sigma_2^{ii}h_{ii}^2+(1+\varepsilon)\frac{\sigma_2^{ii}h_{ii}^2\langle X,\p_i\rangle^2}{u^2}.\nonumber
\end{eqnarray}\par
At $x_0$, differentiate equation (\ref{1.1}) twice,
\begin{eqnarray}\label{e2.6}
\sigma_2^{ii}h_{iik}&=&d_Xf(\p_k)+h_{kk}d_{\nu}f(\p_k),
\end{eqnarray}
and
\begin{eqnarray}\label{e2.7}
\sigma_2^{ii}h_{iikk}+\sigma_2^{pq,rs}h_{pqk}h_{rsk}&\geq& -C-Ch_{11}^2+\sum_lh_{lkk}d_{\nu}f(\p_l),
\end{eqnarray}
where $C$ is a constant under control.

Insert  (\ref{e2.7}) into (\ref{e2.5}),
\begin{eqnarray}
&&\sigma_2^{ii}\phi_{ii}\\
&\geq &\frac{1}{P\log P}[\sum_le^{\kappa_l}(-C-Ch_{11}^2-\sigma_2^{pq,rs}h_{pql}h_{rsl})+\sum_le^{\kappa_l}h_{lkk}d_{\nu}f(\p_l)+2f\sum_le^{\kappa_l}h_{ll}^2\nonumber\\
&&-\sigma_2^{ii}h_{ii}^2\sum_le^{\kappa_l}h_{ll}+\sum_l\sigma_2^{ii}e^{\kappa_l}h_{lli}^2+\sum_{\alpha\neq \beta}\sigma_2^{ii}\frac{e^{\kappa_{\alpha}}-e^{\kappa_{\beta}}}{\kappa_{\alpha}-\kappa_{\beta}}h_{\alpha\beta i}^2-(\frac{1}{P}+\frac{1}{P\log P})\sigma_2^{ii}P_i^2]\nonumber\\
&&- \frac{(1+\varepsilon)\sum_l\sigma_2^{ii}h_{iil}\langle \p_l,X \rangle}{u}+(1+\varepsilon)\sigma_2^{ii}h_{ii}^2+(1+\varepsilon)\frac{\sigma_2^{ii}h_{ii}^2\langle X,\p_i\rangle^2}{u^2}.\nonumber\\
&&+a\kappa_1-Ca\nonumber
\end{eqnarray}
By (\ref{e2.3}) and (\ref{e2.6}),
\begin{eqnarray}
&&\sum_k d_{\nu}f(\p_k)\frac{\sum_le^{\kappa_l}h_{llk}}{P\log P}-\frac{1+\varepsilon}{u}\sum_k\sigma_2^{ii}h_{iik}\langle  \p_k, X\rangle\\
&=&-a\sum_kd_{\nu}f(\p_k)\langle X,\p_k\rangle -\frac{1+\varepsilon}{u}\sum_kd_Xf(\p_k)\langle X,\p_k\rangle.\nonumber
\end{eqnarray}

Denote
\begin{eqnarray}
&&A_i=e^{\kappa_i}(K(\sigma_2)_i^2-\sum_{p\neq q}h_{ppi}h_{qqi}), \ \  B_i=2\sum_{l\neq i}e^{\kappa_l}h_{lli}^2, \ \  C_i=\sigma_2^{ii}\sum_le^{\kappa_l}h_{lli}^2;\nonumber\\
&&D_i=2\sum_{l\neq i}\sigma_2^{ll}\frac{e^{\kappa_l}-e^{\kappa_i}}{\kappa_l-\kappa_i}h_{lli}^2, \ \ E_i=(\frac{1}{P}+\frac{1}{P\log P})\sigma_2^{ii}P_i^2\nonumber.
\end{eqnarray}

Note that $\log P\geq \kappa_1$,
\begin{eqnarray}\label{e2.8}
&&\sigma_2^{ii}\phi_{ii}\\
&\geq &-Ca+(a-C)h_{11}+\frac{1}{P\log P}\sum_{i}e^{\kappa_l}(K(\sigma_2)_l^2-\sum_{p\neq q}h_{ppl}h_{qql}+\sum_{p\neq q}h_{pql}^2)\nonumber\\
&&+\sum_l\sigma_2^{ii}e^{\kappa_l}h_{lli}^2+\sum_{\alpha\neq \beta}\sigma_2^{ii}\frac{e^{\kappa_{\alpha}}-e^{\kappa_{\beta}}}{\kappa_{\alpha}-\kappa_{\beta}}h_{\alpha\beta i}^2-(\frac{1}{P}+\frac{1}{P\log P})\sigma_2^{ii}P_i^2]\nonumber\\
&&+\varepsilon \sigma_2^{ii}h_{ii}^2
+(1+\varepsilon)\frac{\sigma_2^{ii}h_{ii}^2\langle X,\p_i\rangle^2}{u^2}.\nonumber\\
&=& -Ca+(a-C)h_{11}+\frac{1}{P\log P}\sum_{i}(A_i+B_i+C_i+D_i-E_i)\nonumber\\
&&+\varepsilon \sigma_2^{ii}h_{ii}^2
+(1+\varepsilon)\frac{\sigma_2^{ii}h_{ii}^2\langle X,\p_i\rangle^2}{u^2}. \nonumber
\end{eqnarray}

Choose $k=2,l=1$ and $h=i$ in Lemma \ref{lemma 1},
$$
-\sum_{p\neq r}h_{ppi}h_{rri}+(1-\alpha+\frac{\alpha}{\delta})\frac{(\sigma_2)_i^2}{\sigma_2}
\ \ \geq \ \ (\alpha +1-\delta\alpha )\sigma_2[\frac{(\sigma_1)_i}{\sigma_1}]^2\ \ \geq \ \ 0.
$$
Hence, $A_i\geq 0$ if $K$ is sufficiently large.

\begin{lemm}\label{elem 9}
If $$n\kappa_i\leq  \kappa_1,$$  for any  fixed $i\neq 1$ and taking
$\kappa_1$ sufficient large, we have, $$B_i+C_i+D_i-E_i\geq 0.$$
\end{lemm}
\begin{proof}
We have,
$$P_i^2=(e^{\kappa_i}h_{iii}+\sum_{l\neq i}e^{\kappa_l}h_{lli})^2=e^{2\kappa_i}h_{iii}^2+2\sum_{l\neq i}e^{\kappa_i+\kappa_l}h_{lli}h_{iii}+(\sum_{l\neq i}e^{\kappa_l}h_{lli})^2.$$ By the Schwartz inequality,  $$(\sum_{l\neq i}e^{\kappa_l}h_{lli})^2\leq \sum_{l\neq i}e^{\kappa_l}\sum_{l\neq i}e^{\kappa_l}h_{lli}^2.$$ Hence, $$P_i^2\leq e^{2\kappa_i}h_{iii}^2+2\sum_{l\neq i}e^{\kappa_l+\kappa_i}h_{lli}h_{iii}+(P-e^{\kappa_i})\sum_{l\neq i}e^{\kappa_l}h_{lli}^2.$$ In turn,
\begin{eqnarray}\label{e3.14}
&&B_i+C_i+D_i-E_i\\
&\geq&\sum_{l\neq i}(2e^{\kappa_l}+\sigma_2^{ii}e^{\kappa_l}+2\sigma_2^{ll}\frac{e^{\kappa_l}-e^{\kappa_i}}{\kappa_l-\kappa_i})h_{lli}^2+\sigma_2^{ii}e^{\kappa_i}h_{iii}^2-(\frac{1}{P}+\frac{1}{P\log P})\sigma_2^{ii}e^{2\kappa_i}h_{iii}^2\nonumber\\
&&-(\frac{1}{P}+\frac{1}{P\log P})(P-e^{\kappa_i})\sigma_2^{ii}\sum_{l\neq i}e^{\kappa_l}h_{lli}^2-2(\frac{1}{P}+\frac{1}{P\log P})\sigma_2^{ii}\sum_{l\neq i}e^{\kappa_i+\kappa_l}h_{iii}h_{lli}\nonumber\\
&=&\sum_{l\neq i}[(2-\frac{\sigma_2^{ii}}{\log P})e^{\kappa_l}+(\frac{1}{P}+\frac{1}{P\log P})\sigma_2^{ii}e^{\kappa_l+\kappa_i}+2\sigma_2^{ll}\frac{e^{\kappa_l}-e^{\kappa_i}}{\kappa_l-\kappa_i}]h_{lli}^2\nonumber\\
&&+[1-(\frac{1}{P}+\frac{1}{P\log P})e^{\kappa_i}]\sigma_2^{ii}e^{\kappa_i}h_{iii}^2-2(\frac{1}{P}+\frac{1}{P\log P})\sigma_2^{ii}\sum_{l\neq i}e^{\kappa_i+\kappa_l}h_{iii}h_{lli}\nonumber.
\end{eqnarray}
As $$h_{lli}^2+h_{iii}^2\geq 2h_{lli}h_{iii},$$
\begin{eqnarray}\label{e3.15}
&&\sum_{l\neq i,1}(\frac{1}{P}+\frac{1}{P\log P})\sigma_2^{ii}e^{\kappa_l+\kappa_i}h_{lli}^2+\sum_{l\neq i,1}(\frac{1}{P}+\frac{1}{P\log P})\sigma_2^{ii}e^{\kappa_l+\kappa_i}h_{iii}^2\\
&\geq &2(\frac{1}{P}+\frac{1}{P\log P})\sum_{l\neq i,1}\sigma_2^{ii}e^{\kappa_l+\kappa_i}h_{iii}h_{lli}.\nonumber
\end{eqnarray}
Combine \eqref{e3.14} and \eqref{e3.15},
\begin{eqnarray}\label{e3.16}
&&B_i+C_i+D_i-E_i\\
&\geq &\sum_{l\neq i}[(2-\frac{\sigma_2^{ii}}{\log P})e^{\kappa_l}+2\sigma_2^{ll}\frac{e^{\kappa_l}-e^{\kappa_i}}{\kappa_l-\kappa_i}]h_{lli}^2+(\frac{1}{P}+\frac{1}{P\log P})\sigma_2^{ii}e^{\kappa_1+\kappa_i}h_{11i}^2\nonumber\\
&&+[(\frac{1}{P}+\frac{1}{P\log P})e^{\kappa_1}-\frac{1}{\log P}]\sigma_2^{ii}e^{\kappa_i}h_{iii}^2-2(\frac{1}{P}+\frac{1}{P\log P})\sigma_2^{ii}e^{\kappa_i+\kappa_1}h_{iii}h_{11i}\nonumber.\\
&\geq&\sum_{l\neq i}(2-\frac{\sigma_2^{ii}}{\log P})e^{\kappa_l}h_{lli}^2+2\sigma_2^{11}\frac{e^{\kappa_1}-e^{\kappa_i}}{\kappa_1-\kappa_i}h_{11i}^2+\frac{1}{P}\sigma_2^{ii}e^{\kappa_1+\kappa_i}h_{11i}^2\nonumber\\
&&+[\frac{e^{\kappa_1}}{P}-\frac{1}{\log P}]\sigma_2^{ii}e^{\kappa_i}h_{iii}^2-2\frac{1}{P}\sigma_2^{ii}e^{\kappa_i+\kappa_1}h_{iii}h_{11i}\nonumber.
\end{eqnarray}
By the assumptions in the lemma, we have, for $i\geq 2$,
$$2\log P\geq 2\kappa_1\geq \sigma_2^{ii}.$$ Taking $\kappa_1$ sufficient large, we have, $$\frac{e^{\kappa_1}}{2P}\geq \frac{1}{2n}\geq \frac{1}{\log P}.$$
Expanding $e^{x}$ and as $n\kappa_i\leq \kappa_1$,
\begin{eqnarray}
\sigma_2^{11}\frac{e^{\kappa_1}-e^{\kappa_i}}{\kappa_1-\kappa_i}&=&\sigma_2^{11}e^{\kappa_i}\frac{e^{\kappa_1-\kappa_i}-1}{\kappa_1-\kappa_i}=\sigma_2^{11}e^{\kappa_i}\sum_{l=1}^{\infty}\frac{(\kappa_1-\kappa_i)^{l-1}}{l!}\nonumber\\&\geq& \sigma_2^{11}e^{\kappa_i}\frac{(\kappa_1-\kappa_i)^3}{4!} \geq c_0\kappa_1^3\sigma_2^{11}e^{\kappa_i}\geq c_0\kappa_1\sigma_2^{ii}\frac{e^{\kappa_i+\kappa_1}}{P}\nonumber,
\end{eqnarray}
for some positive constant $c_0$. Here, we have used the fact $\kappa_1\sigma_2^{11}\geq 2\sigma_2/n$. The lemma follows from \eqref{e3.16},  previous three inequalities, provided $\kappa_1$ is sufficiently large.
\end{proof}

\begin{lemm}\label{elem 10}
For any index $i$, if $$n\kappa_i\leq  \kappa_1,$$  for any  fixed $j\geq i$ and taking $\kappa_1$ sufficient large, we have, $$B_j+C_j+D_j-(\frac{1}{P}+\frac{2}{n-1}\frac{1}{P\log P})\sigma_2^{jj}P_j^2\geq 0.$$
\end{lemm}
\begin{proof}
Replace the term $\dfrac{1}{P\log P}$ by $\dfrac{2}{n-1}\dfrac{1}{P\log P}$ in the proof of previous lemma, note that
$$2-\frac{2}{n-1}\frac{\sigma_2^{jj}}{\log P}\geq \frac{1}{\kappa_1}(2\kappa_1-\frac{2}{n-1}\sigma_2^{jj})\geq 0.$$ Hence, the arguments in the previous proof can be carried out without further changes.
\end{proof}

\begin{lemm}\label{elem 11}
For any fixed index $j$, if $$n\kappa_j > \kappa_1,$$ we have, for sufficient large $\kappa_1,K$ and sufficient small $\varepsilon$,
$$\frac{1}{P\log P}(A_j+B_j+C_j+D_j-E_j)+(1+\varepsilon)\frac{\sigma_2^{jj}h_{jj}^2\langle X,\p_j\rangle^2}{u^2}\geq 0.$$
\end{lemm}
\begin{proof}
By the Schwarz inequality,
$$\sigma_2^{jj}P_j^2=\sigma^{jj}_2(\sum_le^{\kappa_l}h_{llj})^2\leq \sigma_2^{jj}\sum_le^{\kappa_l}\sum_le^{\kappa_l}h_{llj}^2.$$ Hence,
\begin{eqnarray}\label{e3.17}
C_j-\frac{\sigma_2^{jj}P_j^2}{P}\geq  0.
\end{eqnarray}
By Lemma \ref{exch}, for some sufficient large constant $C$,
$$\sigma_2^{jj}\frac{\kappa_jh^2_{jjj}}{\sigma_1^2}\leq C[K(\sigma_2)_j^2-\sum_{p\neq q}h_{ppj}h_{qqj}+\sum_{l\neq j}h_{llj}^2].$$ Thus,
\begin{eqnarray}\label{e3.18}
\frac{\sigma_2^{jj}P_j^2}{P\log P}&=&\frac{\sigma_2^{jj}}{P\log P}(e^{\kappa_j}h_{jjj}+\sum_{l\neq j}e^{\kappa_l}h_{llj})^2\\
&\leq&\frac{C\sigma_2^{jj}}{P\sigma_1}(e^{2\kappa_i}h_{iii}^2+\sum_{l\neq j}e^{2\kappa_l}h_{llj}^2)\nonumber\\
&\leq &C[\sum_{l\neq j}e^{\kappa_l}h_{llj}^2+\frac{\kappa_j\sigma_2^{jj}}{\sigma_1^2}e^{\kappa_j}h_{jjj}^2]\nonumber\\
&\leq&C(A_j+B_j+e^{\kappa_j}\sum_{l\neq j}h_{llj}^2)\nonumber.
\end{eqnarray}
We \textbf{claim} that $$\sum_{l\neq j}e^{\kappa_l}h^2_{llj}+\sum_{l\neq j}\sigma_2^{ll}\frac{e^{\kappa_l}-e^{\kappa_j}}{\kappa_l-\kappa_j}h_{llj}^2\geq e^{\kappa_j}\sum_{l\neq j}h_{llj}^2.$$
To prove the claim, we divide it two cases.

\noindent Case (A): $\kappa_l> \kappa_j$, obviously, $$ e^{\kappa_l}h^2_{llj}+\sigma_2^{ll}\frac{e^{\kappa_l}-e^{\kappa_j}}{\kappa_l-\kappa_j}h_{llj}^2\geq e^{\kappa_j}h_{llj}^2.$$

\noindent  Case (B): $\kappa_l<\kappa_j$, we have $$\frac{\sigma_2^{ll}}{\kappa_j-\kappa_l}=\frac{\kappa_j-\kappa_l+\sigma_2^{jj}}{\kappa_j-\kappa_l}\geq 1.$$  Therefore, $$ e^{\kappa_l}h^2_{llj}+\sigma_2^{ll}\frac{e^{\kappa_l}-e^{\kappa_j}}{\kappa_l-\kappa_j}h_{llj}^2\geq e^{\kappa_l}h_{llj}^2+(e^{\kappa_j}-e^{\kappa_l})h^2_{llj}= e^{\kappa_j}h_{llj}^2.$$ The claim is verified.   Hence, by \eqref{e3.18} and the claim,
$$\frac{\sigma_2^{jj}P_j^2}{P\log P}\leq c_j(A_j+B_j+D_j).$$ Denote $\delta_j=1/c_j$. It follows from \eqref{e3.17} and \eqref{e2.3} that,
\begin{eqnarray}
&&\frac{1}{P\log P}(A_j+B_j+C_j+D_j-E_j)+(1+\varepsilon)\frac{\sigma_2^{jj}h_{jj}^2\langle X,\p_j\rangle^2}{u^2}\nonumber\\
&\geq& (1+\varepsilon)\frac{\sigma_2^{jj}h_{jj}^2\langle X,\p_j\rangle^2}{u^2}-\frac{1-\delta_j}{(P\log P)^2}\sigma_2^{jj}P_j^2\nonumber\\
&=&(1+\varepsilon)[(1-(1-\delta_j)(1+\varepsilon))\frac{\sigma_2^{jj}h_{jj}^2\langle X,\p_j\rangle^2}{u^2}+2(1-\delta_j)\frac{a\sigma_2^{jj}h_{jj}\langle X,\p_j\rangle^2}{u}]\nonumber\\&&-(1-\delta_j)a^2\sigma_2^{jj}\langle X,\p_j\rangle^2.\nonumber
\end{eqnarray}
The above is nonnegative, if $\kappa_1$ sufficiently large, and $\varepsilon$
is small enough.
\end{proof}

We are in the position to give $C^2$ estimate. We use a similar argument in the previous section. We need to deal with every index in \eqref{e2.8}.  First, we note that $n\kappa_1>\kappa_1$. By Lemma \ref{elem 11},
\begin{eqnarray}\label{e3.19}
\frac{1}{P\log P}(A_1+B_1+C_1+D_1-E_1)+(1+\varepsilon)\frac{\sigma_2^{11}h_{11}^2\langle X,\p_1\rangle^2}{u^2}\geq 0.
\end{eqnarray}
We divide into two different cases.

\par
\noindent Case (A):  Suppose $n\kappa_2\leq \kappa_1$. In this case, we use Lemma \ref{elem 9}. For $i\geq 2$, note that $A_j\geq 0$,
\begin{eqnarray}\label{e3.20}
\frac{1}{P\log P}(A_i+B_i+C_i+D_i-E_i)\geq 0.
\end{eqnarray}
Combine \eqref{e3.19}, \eqref{e3.20} and \eqref{e2.8},
$$\sigma_2^{ii}\phi_{ii}\geq -C+(a-C)\kappa_1.$$  We obtain $C^2$ estimate if $a$ is sufficiently large.
\par
\noindent Case (B): Suppose $n\kappa_2> \kappa_1$. We assume that index $i_0$ satisfies  $n\kappa_{i_0}>\kappa_1$ and $n\kappa_{i_0+1}\leq \kappa_1$. Hence, for index $j\leq i_0$, $n\kappa_j>\kappa_1$. Lemma \ref{elem 11} implies,
\begin{eqnarray}\label{e3.21}
\frac{1}{P\log P}(A_j+B_j+C_j+D_j-E_j)+(1+\varepsilon)\frac{\sigma_2^{jj}h_{jj}^2\langle X,\p_j\rangle^2}{u^2}\geq 0.
\end{eqnarray}
For index $j\geq i_0+1$, by Lemma \ref{elem 10},
\begin{eqnarray}\label{e3.22}
&&\frac{1}{P\log P}(A_j+B_j+C_j+D_j-E_j)+(1+\varepsilon)\frac{\sigma_2^{jj}h_{jj}^2\langle X,\p_j\rangle^2}{u^2}\\
&\geq&-(1-\frac{2}{n-1})\frac{\sigma_2^{jj}P_j^2}{(P\log P)^2}+(1+\varepsilon)\frac{\sigma_2^{jj}h_{jj}^2\langle X,\p_j\rangle^2}{u^2}\nonumber\\
&=&(1+\varepsilon)[(1-\frac{n-3}{n-1}(1+\varepsilon))\frac{\sigma_2^{jj}h_{jj}^2\langle X,\p_j\rangle^2}{u^2}+2\frac{n-3}{n-1}\frac{a\sigma_2^{jj}h_{jj}\langle X,\p_j\rangle^2}{u}]\nonumber\\&&-\frac{n-3}{n-1}a^2\sigma_2^{jj}\langle X,\p_j\rangle^2.\nonumber\\
&\geq &-Ca^2\kappa_1\nonumber.
\end{eqnarray}
The last inequality holds, provided $\varepsilon$ is sufficiently small. Combining \eqref{e3.21}, \eqref{e3.22} and \eqref{e2.8}, we obtain,
$$\sigma_2^{ii}\phi_{ii}\geq -C+(a-C)\kappa_1+\varepsilon\sigma_2^{ii}\kappa_i^2-Ca^2\kappa_1.$$
We further divide the case into two subcases to deal with the above inequality.
\par
\noindent  Case (B1):  Suppose  $\sigma_2^{22}\geq 1$. As $n\kappa_2>\kappa_1$,
\begin{eqnarray}
\sigma_2^{ii}\phi_{ii}&\geq& -C+(a-C)\kappa_1+\varepsilon \sigma_2^{22}\kappa_2^2-Ca^2\kappa_1\nonumber\\
&\geq&-C+(a-C)\kappa_1+\frac{\varepsilon}{n^2} \kappa_1^2-Ca^2\kappa_1\nonumber.
\end{eqnarray}
The above is nonnegative if $\kappa_1$ and $a$ are sufficiently large.
\par
\noindent Case (B2):  Suppose $\sigma_2^{22}< 1$. In this subcase, we may assume that $\kappa_1$ is sufficiently large, then $\kappa_n<0$.  By the assumption,
$1\geq \kappa_1+(n-2)\kappa_n$. This implies,
$$-\kappa_n\geq\frac{\kappa_1-1}{n-2}.$$  Since $\sigma_2^{nn}+\kappa_n= \kappa_1+\sigma_2^{11}$, we have $\sigma^{nn}_2\geq \kappa_1$. Hence,
\begin{eqnarray}
\sigma_2^{ii}\phi_{ii}&\geq& -C+(a-C)\kappa_1+\varepsilon \sigma_2^{nn}\kappa_n^2-Ca^2\kappa_1\nonumber\\
&\geq&-C+(a-C)\kappa_1+\frac{\varepsilon}{(n-2)^2} \kappa_1(\kappa_1-1)^2-Ca^2\kappa_1\nonumber.
\end{eqnarray}
The above is nonnegative, if $a$ and $\kappa_1$ are sufficiently large. The proof of Theorem \ref{theo2 11-0} is complete.

\medskip

We remark that the similar curvature estimate can be established for Dirichlet boundary problem of equation
\begin{equation}\label{2.1-0}
\left\{\begin{matrix}\sigma_2[\kappa(x,u(x))]&=&f(x,u,Du), \\
 u|_{\partial \Omega}&=&\phi,\end{matrix}\right.
\end{equation}
where $\Omega\subset \mathbb R^n$ is a bounded domain. Though such graph over
$\Omega$ may not be starshaped. With the assumption of $C^1$ boundedness, one may shift the origin in $\mathbb R^{n+1}$ in the direction of $E_{n+1}=(0,\cdots,0, 1)$ in appropriate way so that the surface is starshaped with respect to the new origin. Then the proof in this section yields the following theorem, which completely settled the regularity problem considered in Ivochkina \cite{I, I1} when $k=2$.

\begin{theo}\label{theo2 11-0000}
Suppose $u$ is a solution of equation (\ref{2.1-0}), then there is a constant $C$ depending only on $n, k$, $\Omega$, $\|u\|_{C^1}$, $\inf f$ and $\|f\|_{C^2}$, such that
 \begin{equation}\label{M2c2-0}
 \max_{x\in \Omega} |\nabla^2 u(x)| \le C(1+\max_{x\in\partial \Omega} |\nabla^2 u(x)|), \ \ \forall i=1,\cdots, n.\end{equation}
\end{theo}

\section{A global $C^2$ estimate for convex hypersurfaces}
\par
In this section, we consider the global $C^2$ estimates for convex solutions to curvature equation (\ref{1.1}) in $\mathbb{R}^{n+1}$. We need further modify the test function constructed in the previous section.

\begin{theo}\label{theo 11-0}
Suppose $M\subset \mathbb R^{n+1}$ is a convex hypersurface satisfying curvature equation (\ref{1.1}) for some positive function $f(X, \nu)\in C^{2}(\Gamma)$, where $\Gamma$ is an open neighborhood of unit normal bundle of $M$ in $\mathbb R^{n+1} \times \mathbb S^n$, then there is a constant $C$ depending only on $n, k$, $\|M\|_{C^1}$, $\inf f$ and $\|f\|_{C^2}$, such that
 \begin{equation}\label{Mc2-0}
 \max_{X\in M, i=1,\cdots, n} \kappa_i(X) \le C(1+\max_{X\in \partial M, i=1,\cdots, n} \kappa_i(X)).\end{equation}
\end{theo}

\par
To precede, consider the following test function,
\begin{eqnarray}
P(\kappa(X))= \kappa^2_1+\cdots+\kappa_n^2, \quad \phi=\log P(\kappa(X))-2N \log u,
\end{eqnarray}
where $N$ is a constant to be determined later. Note that,
\[\kappa^2_1+\cdots+\kappa_n^2= \sigma_1(\kappa(X))^2-2\sigma_2(\kappa(X)).\]We assume that $\phi$ achieves its maximum value at $x_0\in M$. By a proper
rotation, we may assume that $(h_{ij})$ is a diagonal matrix at the
point, and $h_{11}\geq h_{22}\cdots\geq h_{nn}$.
\par
At $x_0$, differentiate $\phi$ twice,
\begin{eqnarray}\label{3}
\phi_i&=& \frac{\sum_k\kappa_kh_{kki}}{P}-N\frac{u_i}{u}\\
&=&\frac{\sum_k\kappa_kh_{kki}}{P}-N\frac{h_{ii}\langle \p_i,X \rangle}{u}\ \  =   \  \ 0,\nonumber
\end{eqnarray}
and,
\begin{eqnarray}\label{4}
0&\geq &\frac{1}{P}[\sum_k\kappa_kh_{kk,ii}+\sum_{k}h_{kki}^2+\sum_{p\neq q}h_{pqi}^2]-\frac{2}{P^2}[\sum_k\kappa_kh_{kki}]^2\\
&&-N\frac{u_{ii}}{u} +N\frac{u_i^2}{u^2}\nonumber\\
&=&\frac{1}{P}[\sum_k\kappa_k(h_{ii,kk}+(h_{ik}^2-h_{ii}h_{kk})h_{ii}+(h_{ii}h_{kk}-h_{ik}^2)h_{kk})\nonumber\\
&&+\sum_{k}h_{kki}^2+\sum_{p\neq q}h_{pqi}^2]-\frac{2}{P^2}[\sum_k\kappa_kh_{kki}]^2-N\frac{\sum_lh_{ii,l}\langle X,\p_l\rangle}{u}\nonumber\\
&&-N\frac{h_{ii}}{u} +Nh_{ii}^2+N\frac{h_{ii}^2\langle \p_i,X\rangle^2}{u^2}.\nonumber
\end{eqnarray}
\par
Now differentiate equation (\ref{1.1}) twice,
\begin{eqnarray}\label{6}
\sigma_k^{ii}h_{iij}&=& d_Xf(X_j) + d_{\nu} f( \nu_j)\ \ = \ \  d_X f(\p_j)+h_{jj}d_{\nu}f(\p_j),
\end{eqnarray}
\begin{eqnarray}\label{7}
&&\sigma_k^{ii}h_{iijj}+\sigma_k^{pq,rs}h_{pqj}h_{rsj}\\
&=&d_Xf(X_{jj})+d^2_{X}f(X_j,X_j)+2d_Xd_{\nu}f(X_j,\nu_j)
+d^2_{\nu}f(\nu_j,\nu_j)+d_{\nu}f(\nu_{jj}).\nonumber\\
&=&-h_{jj}d_Xf(\nu)+d^2_{X}f(\p_j,\p_j)+2h_{jj}d_Xd_{\nu}f(\p_j,\p_j)
+h_{jj}^2d^2_{\nu}f(\p_j,\p_j)\nonumber\\&&+\sum_kh_{kjj}d_{\nu}f(\p_k)-h_{jj}^2d_{\nu}f(\nu)\nonumber\\
&\geq&-C-C\kappa_j-C\kappa_j^2+\sum_kh_{kjj}d_{\nu}f(\p_k)\nonumber\\
&\geq&-C-C\kappa_j^2+\sum_kh_{kjj}d_{\nu}f(\p_k)\nonumber.
\end{eqnarray}
The Schwarz inequality is used in the last inequality.

Since
$$
-\sigma_k^{pq,rs}h_{pql}h_{rsl}\ \ =\  \ -\sigma_k^{pp,qq}h_{ppl}h_{qql}+\sigma_k^{pp,qq}h_{pql}^2,
$$
it follows from (\ref{3}) and (\ref{6}),
\begin{eqnarray}
\frac{1}{P}\sum_{l,s}\kappa_lh_{sll}d_{\nu}f(\p_s)-\frac{N\sigma_k^{ii}\sum_sh_{iis}\langle \p_s, X\rangle }{u}
& =& -\frac{N}{u}\sum_sd_Xf(\p_s)\langle \p_s,X\rangle.
\end{eqnarray}

Denote
\begin{eqnarray}
&& A_i= \frac{\kappa_i}{P}(K(\sigma_k)_i^2-\sum_{p,q}\sigma_k^{pp,qq}h_{ppi}h_{qqi}), \ \  B_{i}=2\sum_j\frac{\kappa_j}{P}\sigma_k^{jj,ii}h_{jji}^2, \nonumber \\
&& C_i=2\sum_{j\neq i}\frac{\sigma_k^{jj}}{P}h_{jji}^2, \ \
 D_i=\frac{1}{P}\sum_j\sigma_k^{ii}h_{jji}^2,\  \
E_i=\frac{2\sigma_k^{ii}}{P^2}(\sum_j \kappa_jh_{jji})^2.\nonumber
\end{eqnarray}

Contract with $\sigma_k^{ii}$ in both side of inequality (\ref{4}), by (\ref{6}) and  (\ref{7}),
\begin{eqnarray}\label{3.8} 0&\geq&\frac{1}{P}[\sum_l\kappa_l(-C-C\kappa_{l}^2+\sum_sh_{sll}d_{\nu}f(\p_s)-K(\sigma_k)_l^2+K(\sigma_k)_l^2-\sigma_k^{pq,rs}h_{pql}h_{rsl})\\
&&+\sigma_k^{ii}h_{ii}\sum_l\kappa_l^3-\sigma_k^{ii}h_{ii}^2\sum_l\kappa_l^2+\sum_l\sigma_k^{ii}h_{lli}^2+\sigma_k^{ii}\sum_{p,q}h_{pqi}^2]
-\frac{2\sigma_k^{ii}}{P^2}(\sum_j \kappa_jh_{jji})^2\nonumber\\
&&-N\frac{\sum_l\sigma_k^{ii}h_{ii,l}\langle X,\p_l\rangle}{u}
-N\frac{\sigma_k^{ii}h_{ii}}{u} +N\sigma_k^{ii}h_{ii}^2+N\frac{\sigma_k^{ii}h_{ii}^2\langle \p_i,X\rangle^2}{u^2}\nonumber\\
&\geq&\frac{1}{P}[\sum_l\kappa_l(-C-C\kappa_{l}^2-K(\sigma_k)_l^2+K(\sigma_k)_l^2
-\sigma_k^{pp,qq}h_{ppl}h_{qql}+\sigma_k^{pp,qq}h_{pql}^2)\nonumber\\
&&+kf\sum_l\kappa_l^3-\sigma_k^{ii}h_{ii}^2\sum_l\kappa_l^2+\sum_l\sigma_k^{ii}h_{lli}^2+\sigma_k^{ii}\sum_{p,q}h_{pqi}^2]
-\frac{2\sigma_k^{ii}}{P^2}(\sum_j \kappa_jh_{jji})^2
\nonumber\\&&-N\frac{kf}{u}+N\sigma_k^{ii}h_{ii}^2+N\frac{\sigma_k^{ii}h_{ii}^2\langle
\p_i,X\rangle^2}{u^2}-\frac{N}{u}\sum_sd_Xf(\p_s)\langle
\p_s,X\rangle\nonumber\\
&\ge& \frac{1}{P}[\sum_l\kappa_l(-C-C\kappa_{l}^2+\sum_sh_{sll}d_{\nu}f(\p_s)-K(\sigma_k)_l^2)
+\sigma_k^{ii}h_{ii}\sum_l\kappa_l^3-\sigma_k^{ii}h_{ii}^2\sum_l\kappa_l^2]\nonumber\\
&&-N\frac{kf}{u}+N\sigma_k^{ii}h_{ii}^2+N\frac{\sigma_k^{ii}h_{ii}^2\langle
\p_i,X\rangle^2}{u^2}-\frac{N}{u}\sum_sd_Xf(\p_s)\langle
\p_s,X\rangle\nonumber\\
&& +\sum_{i} (A_i+B_i+C_i+D_i-E_i).\nonumber
\end{eqnarray}

The main part of the proof is to deal with the third order derivatives.
We divide it to two cases:
\begin{enumerate} \item $i\neq 1$; \item $i=1$.\end{enumerate}
\par

\begin{lemm} \label{lemma 2}
For $i\neq 1$,   if  $$\sqrt{3}\kappa_i\ \ \leq \ \ \kappa_1,$$ we have,
$$A_i+B_i+C_i+D_i-E_i\ \ \geq \ \  0.$$
\end{lemm}
\begin{proof}
\par
By (\ref{1.5}) in Lemma \ref{lemma 1} and note that $\sigma_1^{pp,qq}=0$,  when $K$ is  sufficiently large,
\begin{eqnarray}\label{10}
K(\sigma_k)_i^2-\sigma_k^{pp,qq}h_{ppi}h_{qqi}&\geq&\sigma_k(1+\frac{\al}{2})[\frac{(\sigma_1)_i}{\sigma_1}]^2\ \ > \  \  0.
\end{eqnarray}

\begin{eqnarray}\label{3.10}
&&P^2(B_i+C_i+D_i-E_i)\\
&=&\sum_{j\neq i}P(2\kappa_j\sigma_k^{jj,ii}+2\sigma_k^{jj}+\sigma_k^{ii})h_{jji}^2+P\sigma_k^{ii}h_{iii}^2
-2\sigma_k^{ii}(\sum_{j\neq i}\kappa_j^2h_{jji}^2+\kappa_i^2h_{iii}^2\nonumber\\&&+\sum_{k\neq l}\kappa_k\kappa_lh_{kki}h_{lli})\nonumber\\
&=&\sum_{j\neq i}[P(3\sigma_k^{ii}+2\sigma_k^{jj}-2\sigma_{k-1}(\kappa|ij))-2\sigma_k^{ii}\kappa_j^2]h_{jji}^2+(P-2\kappa_i^2)\sigma_k^{ii}h_{iii}^2\nonumber\\
&&-2\sigma_k^{ii}\sum_{k\neq l}\kappa_k\kappa_lh_{kki}h_{lli}\nonumber\\
&=&\sum_{j\neq i}(P+2(P-\kappa_j^2))\sigma_k^{ii}h_{jji}^2+(P-2\kappa_i^2)\sigma_k^{ii}h_{iii}^2
-2\sigma_k^{ii}\sum_{k\neq l}\kappa_k\kappa_lh_{kki}h_{lli}\nonumber\\
&&+2P\sum_{j\neq i}\kappa_i\sigma^{jj,ii}h_{jji}^2\nonumber.
\end{eqnarray}
Note that,
\begin{eqnarray}\label{3.11}
2\sum_{j\neq i}\sum_{k\neq i,j}\kappa_k^2h_{jji}^2&=& \sum_{l\neq i}\sum_{k\neq i,l}\kappa_k^2h_{lli}^2+\sum_{k\neq i}\sum_{l\neq i,k}\kappa_l^2h_{kki}^2\\
  & \geq& 2\sum_{k\neq l; k,l\neq i}\kappa_k\kappa_lh_{kki}h_{lli}.\nonumber
\end{eqnarray}
By $\sqrt{3}\kappa_i\leq\kappa_1$ or $\kappa_1^2\geq 3\kappa_i^2$,
\begin{eqnarray}\label{3.12}
\sum_{j\neq i,1}(\frac{2P}{3}+2\kappa_i^2)h_{jji}^2+\sum_{j\neq i,1}\kappa_j^2h_{iii}^2\ \ \geq \ \ 2\kappa_ih_{iii}\sum_{j\neq i,1}\kappa_jh_{jji} .
\end{eqnarray}
Then (\ref{3.10}) becomes,
\begin{eqnarray}
&&P^2(B_i+C_i+D_i-E_i)\\
&\geq&(P+2\kappa_i^2)\sigma_k^{ii}h_{11i}^2+(\kappa_1^2-\kappa_i^2)\sigma_k^{ii}h_{iii}^2
-4\sigma_k^{ii}\kappa_ih_{iii}\kappa_1h_{11i}\nonumber\\
&&+\frac{P}{3}\sum_{j\neq 1,i}\sigma_k^{ii}h_{jji}^2+2P\sum_{j\neq i}\kappa_i\sigma^{jj,ii}h_{jji}^2\nonumber\\
&\geq&\sigma_k^{ii}[(\kappa_1^2+3\kappa_i^2)h^{2}_{11i}+(\kappa_1^2-\kappa_i^2)h_{iii}^2-4\kappa_1\kappa_ih_{iii}h_{11i}]+2P\sum_{j\neq i}\kappa_i\sigma^{jj,ii}h_{jji}^2.\nonumber
\end{eqnarray}
The above is nonnegative, provided the following inequality holds,
\begin{eqnarray}\label{newabove}
\sqrt{\kappa_1^2+3\kappa_i^2}\sqrt{\kappa_1^2-\kappa_i^2}&\geq& 2\kappa_1\kappa_i.
\end{eqnarray}
Set $x=\kappa_i/\kappa_1$. Inequality (\ref{newabove}) is equivalent to the following inequality,
$$3x^4+2x^2-1\leq 0.$$  This follows from the condition $\kappa_1\geq \sqrt{3}\kappa_i$. The proof is complete.
\end{proof}

We need another Lemma.
\begin{lemm}\label{lemma 3}
For $\lambda=1,\cdots, k-1$, if there exists some positive constant $\delta\leq 1 $, such that $\kappa_{\lambda}/\kappa_1\geq \delta$. Then there exits a sufficient small positive constant
$\delta'$ depending on $\delta$,  such that, if  $\kappa_{\lambda+1}/\kappa_1\leq \delta'$,  we have
$$A_i+B_i+C_i+D_i-E_i\geq 0,$$ for $i=1,\cdots, \lambda$.
\end{lemm}
\begin{proof} By (\ref{3.10}) and (\ref{3.11}),
\begin{eqnarray}
&&P^2(B_i+C_i+D_i-E_i)\\
&=&\sum_{j\neq i}[P(3\sigma_k^{ii}+2\sigma_k^{jj}-2\sigma_{k-1}(\kappa|ij))-2\sigma_k^{ii}\kappa_j^2]h_{jji}^2+(P-2\kappa_i^2)\sigma_k^{ii}h_{iii}^2\nonumber\\
&&-2\sigma_k^{ii}\sum_{k\neq l}\kappa_k\kappa_lh_{kki}h_{lli}\nonumber\\
&\geq&\sum_{j\neq i}(P+2\kappa_i^2)\sigma_k^{ii}h_{jji}^2+(P-2\kappa_i^2)\sigma_k^{ii}h_{iii}^2
-4\sigma_k^{ii}\kappa_ih_{iii}\sum_{j\neq i}\kappa_jh_{jji}\nonumber\\
&&+P\sum_{j\neq i}2(\sigma_{k-1}(\kappa|j)-\sigma_{k-1}(\kappa|ij))h_{jji}^2.\nonumber
\end{eqnarray}
For $i=1$, the above inequality becomes,
\begin{eqnarray}\label{3.151}
&&P^2(B_i+C_i+D_i-E_i)\\
&\geq &\sum_{j\neq 1}(3\kappa_1^2\sigma_k^{11}+\kappa_1^2\sigma_k^{jj})h_{jj1}^2+\sum_{j\neq 1}\kappa_j^2\sigma^{11}_kh_{111}^2-4\sigma_{k}^{11}\kappa_1h_{111}\sum_{j\neq 1}\kappa_j h_{jj1}\nonumber\\
&&+P\sum_{j\neq 1}(\sigma_{k-1}(\kappa|j)-2\sigma_{k-1}(\kappa|1j))h_{jj1}^2-\kappa_1^2\sigma_k^{11}h_{111}^2\nonumber\\
&\geq&P\sum_{j\neq 1}(\sigma_{k-1}(\kappa|j)-2\sigma_{k-1}(\kappa|1j))h_{jj1}^2-\kappa_1^2\sigma_k^{11}h_{111}^2\nonumber.
\end{eqnarray}
For $i\neq 1$, we replace the index $j\neq i,1$ with $j\neq i$ in (\ref{3.12}), then
\begin{eqnarray}\label{3.152}
 P^2(B_i+C_i+D_i-E_i) \geq \ P\sum_{j\neq i}2(\sigma_{k-1}(\kappa|j)-\sigma_{k-1}(\kappa|ij))h_{jji}^2-\kappa_i^2\sigma_k^{ii}h_{iii}^2.
 \end{eqnarray}
By (\ref{1.5}) in Lemma \ref{lemma 1},
\begin{eqnarray}\label{3.16}
A_i&\geq &\frac{\kappa_i}{P}[\sigma_k(1+\frac{\alpha}{2})\frac{(\sigma_{\lambda})^2_i}{\sigma^2_{\lambda}}-\frac{\sigma_k}{\sigma_{\lambda}}\sigma_{\lambda}^{pp,qq}h_{ppi}h_{qqi}]\\
&\geq&\frac{\kappa_i\sigma_k}{P\sigma_{\lambda}^2}[(1+\frac{\alpha}{2})\sum_{a}(\sigma_{\lambda}^{aa}h_{aai})^2+\frac{\alpha}{2}\sum_{a\neq b}\sigma_{\lambda}^{aa}\sigma_{\lambda}^{bb}h_{aai}h_{bbi}\nonumber\\&&+\sum_{a\neq b}(\sigma_{\lambda}^{aa}\sigma_{\lambda}^{bb}-\sigma_{\lambda}\sigma_{\lambda}^{aa,bb})h_{aai}h_{bbi}].\nonumber
\end{eqnarray}
For $\lambda=1$,  note that $\sigma_1^{aa}=1$ and $\sigma_1^{aa,bb}=0$. Hence,
\begin{eqnarray}
(1+\frac{\alpha}{2})\sum h_{aai}h_{bbi}
&\geq&2(1+\frac{\alpha}{2})\sum_{a\neq 1}h_{aai}h_{11i}+(1+\frac{\alpha}{2})h_{11i}^2\\
&\geq&(1+\frac{\alpha}{4})h_{11i}^2-C_{\alpha}\sum_{a\neq
1}h_{aai}^2\nonumber
\end{eqnarray}
In turn,
\begin{eqnarray}\label{3.18}
P^2A_i&\geq&\frac{P\kappa_i\sigma_k}{\sigma_1^2}(1+\frac{\alpha}{4})h^2_{11i}-\frac{\kappa_i PC_{\alpha}}{\sigma_1^2}\sum_{a\neq 1}h_{aai}^2\\
&\geq&\frac{\kappa_i^2\sigma_k^{ii}}{(1+\sum_{j\neq 1}\kappa_j/\kappa_1)^2}(1+\frac{\alpha}{4})h^2_{11i}-C_{\alpha}\kappa_i\sum_{a\neq 1}h_{aai}^2\nonumber\\
&\geq&\kappa_i^2\sigma_k^{ii}h_{11i}^2-C_{\alpha}\kappa_i\sum_{a\neq 1}h_{aai}^2.\nonumber
\end{eqnarray}
The last inequality comes from the fact
\begin{eqnarray}
1+\frac{\alpha}{4}&\geq& (1+(n-1)\delta')^2.
\end{eqnarray}
For $\lambda\geq 2$, obviously, for $a\neq b$,
\begin{eqnarray}\label{3.20}
&&\ \ \ \ \ \sigma_{\lambda}^{aa}\sigma_{\lambda}^{bb}-\sigma_{\lambda}\sigma_{\lambda}^{aa,bb}\\
&=&(\kappa_b\sigma_{\lambda-2}(\kappa|ab)+\sigma_{\lambda-1}(\kappa|ab))(\kappa_a\sigma_{\lambda-2}(\kappa|ab)+\sigma_{\lambda-1}(\kappa|ab))\nonumber\\
&&-(\kappa_a\kappa_b\sigma_{\lambda-2}(\kappa|ab)+\kappa_a\sigma_{\lambda-1}(\kappa|ab)+\kappa_b\sigma_{\lambda-1}(\kappa|ab)+\sigma_{\lambda}(\kappa|ab))\sigma_{\lambda-2}(\kappa|ab) \nonumber \\
&=&\sigma_{\lambda-1}^2(\kappa|ab)-\sigma_{\lambda}(\kappa|ab)\sigma_{\lambda-2}(\kappa|ab)\nonumber\\
&\geq &0,\nonumber
\end{eqnarray}
by the Newton inequality. It follows from (\ref{3.20}),
\begin{eqnarray}\label{3.21}
&&\sum_{a\neq b; a,b\leq \lambda}(\sigma_{\lambda}^{aa}\sigma_{\lambda}^{bb}-\sigma_{\lambda}\sigma_{\lambda}^{aa,bb})h_{aai}h_{bbi}\\
&\geq &-\sum_{a\neq b; a,b\leq \lambda}(\sigma_{\lambda-1}^2(\kappa|ab)-\sigma_{\lambda}(\kappa|ab)\sigma_{\lambda-2}(\kappa|ab))h_{aai}^2\nonumber\\
&\geq&-\sum_{a\neq b;a,b\leq \lambda}C_1(\frac{\kappa_{\lambda+1}}{\kappa_b})^2(\sigma_{\lambda}^{aa}h_{aai})^2\nonumber\\
&\geq&-\frac{C_2}{\delta^2}(\frac{\kappa_{\lambda+1}}{\kappa_1})^2\sum_{a}(\sigma^{aa}_{\lambda}h_{aai})^2\ \ \geq \ \ -\epsilon\sum_{a}(\sigma_{\lambda}^{aa}h_{aai})^2\nonumber.
\end{eqnarray}
Here, we choose a sufficient small $\delta'$, such that,
\begin{eqnarray}\label{3.22}
\delta'&\leq& \delta\sqrt{\epsilon/C_2}.
\end{eqnarray}
By (\ref{3.20}),
\begin{eqnarray}\label{3.23}
&&2\sum_{a\leq \lambda; b> \lambda}(\sigma_{\lambda}^{aa}\sigma_{\lambda}^{bb}-\sigma_{\lambda}\sigma_{\lambda}^{aa,bb})h_{aai}h_{bbi}\\
&\geq &-2\sum_{a\leq \lambda; b> \lambda}\sigma_{\lambda}^{aa}\sigma_{\lambda}^{bb}|h_{aai}h_{bbi}|\nonumber\\
&\geq&-\epsilon\sum_{a\leq \lambda; b>\lambda}(\sigma_{\lambda}^{aa}h_{aai})^2-\frac{1}{\epsilon}\sum_{a\leq \lambda; b>\lambda}(\sigma_{\lambda}^{bb}h_{bbi})^2\nonumber.
\end{eqnarray}
Again by (\ref{3.20}),
\begin{eqnarray}\label{3.24}
\sum_{a\neq b; a,b> \lambda}(\sigma_{\lambda}^{aa}\sigma_{\lambda}^{bb}-\sigma_{\lambda}\sigma_{\lambda}^{aa,bb})h_{aai}h_{bbi}&\geq& -\sum_{a\neq b; a,b>\lambda}\sigma_{\lambda}^{aa}\sigma_{\lambda}^{bb}|h_{aai}h_{bbi}|\\
&\geq &-\sum_{a\neq b; a,b>\lambda}(\sigma_{\lambda}^{aa}h_{aai})^2.\nonumber
\end{eqnarray}
Combining (\ref{3.16}), (\ref{3.21}), (\ref{3.23}) and (\ref{3.24}), taking $\alpha=0$ in (\ref{3.16}), we get,
\begin{eqnarray}
A_i
&\geq&\frac{\kappa_i\sigma_k}{P\sigma_{\lambda}^2}[(1-2\epsilon)\sum_{a\leq \lambda}(\sigma_{\lambda}^{aa}h_{aai})^2-C_{\epsilon}\sum_{a>\lambda}(\sigma_{\lambda}^{aa}h_{aai})^2].
\end{eqnarray}
Therefore,
\begin{eqnarray}\label{3.26}
&&P^2A_i\\&\geq&\frac{P\kappa^2_i\sigma^{ii}_k}{\sigma_{\lambda}^2}(1-2\epsilon)\sum_{a\leq \lambda}(\sigma_{\lambda}^{aa}h_{aai})^2-\frac{P\kappa_i\sigma_k C_{\epsilon}}{\sigma_{\lambda}^2}\sum_{a>{\lambda}}(\sigma_{\lambda}^{aa}h_{aai})^2\nonumber\\
&\geq&\frac{P\kappa^2_i\sigma_k^{ii}}{\kappa_1^2}(1-2\epsilon)\sum_{a\leq \lambda}(\frac{\kappa_a\sigma_{\lambda}^{aa}}{\sigma_{\lambda}})^2h_{aai}^2-\frac{\kappa_1^2\kappa_i C_{\epsilon}}{\sigma_{\lambda}^2}\sum_{a>\lambda}(\sigma_{\lambda}^{aa}h_{aai})^2\nonumber\\
&\geq&\kappa^2_i\sigma_k^{ii}(1-2\epsilon)(1+\delta^2)\sum_{a\leq \lambda}(1-\frac{C_3\kappa_{\lambda+1}}{\kappa_a})^2h_{aai}^2-\frac{\kappa_a^2\kappa_i C_{\epsilon}}{\delta^2\sigma_{\lambda}^2}\sum_{a>\lambda}(\sigma_{\lambda}^{aa}h_{aai})^2\nonumber\\
&\geq&\kappa^2_i\sigma_k^{ii}(1-2\epsilon)(1+\delta^2)(1-\frac{C_3\kappa_{\lambda+1}}{\delta\kappa_1})^2\sum_{a\leq \lambda}h_{aai}^2-\frac{\kappa_i C_{\epsilon}}{\delta^2}\sum_{a>\lambda}h_{aai}^2\nonumber\\
&\geq&\kappa_i^2\sigma_k^{ii}\sum_{a\leq \lambda}h_{aai}^2-\frac{\kappa_i C_{\epsilon}}{\delta^2}\sum_{a>\lambda}h_{aai}^2,\nonumber
\end{eqnarray}
in the above, we have used the fact that we may choose $\delta'$ and $\epsilon$ satisfying
\begin{eqnarray}
\delta'C_3\leq 2\epsilon \delta ,\ \  \   (1-2\epsilon)^2(1+\delta^2)\geq 1 .
\end{eqnarray}

By (\ref{3.151}), (\ref{3.152}), (\ref{3.18}) and (\ref{3.26}), for each $i$,  we have,
\begin{eqnarray}
&&P^2(A_i+B_i+C_i+D_i-E_i)\\
&\geq& \sum_{j\neq i}(P\sigma_{k-1}(\kappa|j)-2P\sigma_{k-1}(\kappa|ij))h_{jji}^2-C_{\alpha,\delta}\kappa_i\sum_{j>\lambda}h^2_{jji}.\nonumber
\end{eqnarray}
Now, for $j\leq \lambda$,
\begin{eqnarray}\label{3.31}
\sigma_{k-1}(\kappa|j)-2\sigma_{k-1}(\kappa|ij)
&=& \kappa_i\sigma_{k-2}(\kappa|ij)-\sigma_{k-1}(\kappa|ij)\\
&\geq&\frac{\kappa_1\cdots\kappa_k}{\kappa_j}-C\frac{\kappa_1\cdots\kappa_{k+1}}{\kappa_i\kappa_j}\nonumber\\
&\geq&\frac{\kappa_1\cdots\kappa_k}{\kappa_j}(1-C\frac{\kappa_{k+1}}{\delta\kappa_1})\nonumber\\
&\geq&\frac{\varepsilon \sigma_k}{\kappa_j}(1-C_4\delta'/\delta).\nonumber
\end{eqnarray}
For $\lambda<j\leq k$, in a similar way, we have,
\begin{eqnarray}\label{3.32}
\sigma_{k-1}(\kappa|j)-2\sigma_{k-1}(\kappa|ij)-C_{\epsilon, \alpha}\frac{\kappa_i}{P}
&\geq&\frac{\varepsilon \sigma_k}{\kappa_j}(1-C_4\delta'/\delta)-\frac{C_{\epsilon,\alpha}}{\kappa_1}.
\end{eqnarray}
For $j>k$,
\begin{eqnarray}\label{3.33}
&&\sigma_{k-1}(\kappa|j)-2\sigma_{k-1}(\kappa|ij)-\frac{C_{\epsilon,\alpha}\kappa_i}{\kappa^2_1} \\
&=& \kappa_i\sigma_{k-2}(\kappa|ij)-\sigma_{k-1}(\kappa|ij)-\frac{C_{\epsilon,\alpha}\kappa_i}{\kappa^2_1}
\ \ \geq \ \ \frac{\kappa_1\cdots\kappa_k}{\kappa_k}-C\frac{\kappa_1\cdots\kappa_{k}}{\kappa_i}-\frac{C_{\epsilon,\alpha}}{\kappa_1}\nonumber\\
&\geq&\frac{\kappa_1\cdots\kappa_k}{\kappa_k}(1-C\frac{\kappa_{k}}{\delta\kappa_1})-\frac{C_{\epsilon,\alpha}}{\kappa_1}
\ \ \geq \ \ \frac{\varepsilon \sigma_k}{\kappa_k}(1-C_4\delta'/\delta)-\frac{C_{\epsilon,\alpha}}{\kappa_1}.\nonumber
\end{eqnarray}
We may choose $$\delta'\ \ \leq \ \ \delta/(2C_4), $$ so that (\ref{3.31}) is nonnegative. We further impose that
$$\delta'\ \ \leq \ \ \varepsilon\sigma_k/(2C_{\epsilon,\alpha}).$$ Thus, both (\ref{3.32}) and (\ref{3.33}) are non-negative.  The proof is complete.
\end{proof}

A directly corollary of Lemma \ref{lemma 2} and Lemma \ref{lemma 3} is the following.
\begin{coro}\label{cor4}
There exists a finite sequence of positive numbers $\{\delta_i\}_{i=1}^k$, such that, if the following inequality holds for some $1\leq i\leq k$, $$\frac{\kappa_i}{\kappa_1} \ \ \geq \ \ \delta_i,$$ then,
\begin{eqnarray}\label{3.34}
0&\leq&\frac{1}{P}[\sum_l\kappa_l(K(\sigma_k)_l^2-\sigma_k^{pp,qq}h_{ppl}h_{qql}+\sigma_k^{pp,qq}h_{pql}^2)+\sum_{p,q}\sigma_k^{ii}h_{pqi}^2]\\
&&-\frac{2\sigma_k^{ii}}{P^2}(\sum_j \kappa_jh_{jji})^2\nonumber.
\end{eqnarray}
\end{coro}
\begin{proof} We use induction to find sequence $\{\delta_i\}_{i=1}^k$. Let $\delta_1=1/\sqrt{3}$. Then $\kappa_1/\kappa_1=1>\delta_1$. The claim holds for $i=1$ follows from the proof in the previous lemma. Assume that we have determined $\delta_i$ for $1\leq i\leq k-1$. We want to search for $\delta_{i+1}$.  In Lemma \ref{lemma 3}, we may choose $\lambda=i$ and $\delta=\delta_i$. Then there is some $\delta'_{i+1}$ such that, if $\kappa_{i+1}\leq \delta'_{i+1}\kappa_1$, we have $A_j+B_j+C_j+D_j-E_j \geq 0$ for $1\leq j\leq i$. Pick $$\delta_{i+1}=\min\{\delta_1, \delta_{i+1}'\}.$$ If $\kappa_{i+1}\leq \delta_{i+1}\kappa_1$, by Lemma \ref{lemma 2},  $A_j+B_j+C_j+D_j-E_j \geq 0$ for $j\geq i+1$. We obtain (\ref{3.34}) for $i+1$ case.
\end{proof}

\medskip

\noindent{\bf Proof of Theorem \ref{theo 11-0}.} Again, the proof will be divided into two cases.

\noindent Case (A): There exists some $2\leq i\leq k$, such that  $\kappa_i\leq \delta_i\kappa_1$. By Corollary \ref{cor4}, (\ref{3.8}),(\ref{6}) and the Schwarz inequality,
\begin{eqnarray}
 0&\geq&\frac{1}{P}[\sum_l\kappa_l(-C-C\kappa_{l}^2-K(\sigma_k)_l^2)+kf\sum_l\kappa_l^3-\sigma_k^{ii}h_{ii}^2\sum_l\kappa_l^2]
-N\frac{kf}{u} \nonumber\\&&+N\sigma_k^{ii}h_{ii}^2+N\frac{\sigma_k^{ii}h_{ii}^2\langle \p_i,X\rangle^2}{u^2}-\frac{N}{u}\sum_sd_Xf(\p_s)\langle \p_s,X\rangle.\nonumber\\
&\geq&\frac{1}{P}[-C(K)-C(K)\sum_l\kappa_l^3]-\sigma_k^{ii}h_{ii}^2
+N\sigma_k^{ii}h_{ii}^2-C(N)\nonumber\\
&\geq&-\frac{C(K)\kappa_1^3+C(K)}{P}+(N-1)\varepsilon\sigma_k\kappa_1-C(N),\nonumber
\end{eqnarray}
in the last inequality, we have used $$\kappa_1\sigma_k^{11}\geq \frac{k}{n}\sigma_k.$$
If we choose $$\varepsilon \sigma_k(N-1)\geq C(K)+1,$$  an upper bound of $\kappa_1$ follows.

\noindent Case(B): If the Case(A) does not hold.  That means $\kappa_k\geq \delta_k\kappa_1$. Since $\kappa_l\geq 0$, we have,
$$\sigma_k\geq \kappa_1\kappa_2\cdots\kappa_k\geq\delta_k^k\kappa_1^k.$$ This implies the bound of $\kappa_1$.  \qed

\medskip

We have three remarks about the above $C^2$  estimate.
\begin{rema}
Following the same arguments, we can establish similar $C^2$ estimates for convex solutions of $\sigma_k$-Hessian equation
\begin{equation} \sigma_k(\nabla^2 u)=f(x,u, \nabla u).\end{equation}
\end{rema}
\begin{rema}
The function $P$ chosen here is the order $2$ Newton polynomial. In fact, our arguments can be adopted for Newton polynomials of order $m\geq 2$ to obtain the global $C^2$ estimate.
\end{rema}
\begin{rema}
The assumption of convexity of solution can be weakened. Our proof works if the principal curvatures are bounded from below by some constant, with test function modified as $\log P+g(u)+a|X|^2$. The convexity assumption can also be weakened to $k+1$ convex.
\end{rema}

\section{The prescribed curvature equations}

The a priori estimates we establish in the previous sections may yield existence of solutions to the prescribing equation (\ref{1.1}). By Theorem \ref{theo 11} and Theorem \ref{theo2 11}, we need to obtain $C^1$ bounds for solutions. The treatment of this section follows largely from Caffarelli-Nirenberg-Spruck \cite{CNS5}. We are looking for starshaped hypersurface $M$.

For $x\in\mathbb{S}^n$,  let
$$X(x)=\rho(x)x,$$ be the position vector of the hypersurface $M$.

First is the gradient bound.

\begin{lemm}\label{lemma 12}
If the hypersurface $X$ satisfies condition (\ref{4.2}) and $\rho$ has positive upper and lower bound, then  there is a constant $C$  depending on the minimum and maximum values of $\rho$, such that, $$|\nabla \rho|\leq C.$$
\end{lemm}
\begin{proof} We only need to obtain a positive lower bound of $u$.
Following \cite{GLM}, we consider
$$\phi=-\log u+\gamma(|X|^2).$$
Assume $X_0$ is the maximum value point of $\phi$. If $X$ is not parallel to the normal direction of $X$ at $X_0$ , we may choose the local orthonormal frame $\{e_1,\cdots, e_n\}$ on $M$ satisfying
$$\langle X, e_1\rangle \neq 0, \ \  \text{ and }\ \ \langle X,e_i\rangle=0,\ \  i\geq 2.$$
  Then, at $X_0$,
\begin{eqnarray}\label{4.3}
 u_i&=&2u\gamma'\langle X,e_i\rangle, \\
 \phi_{ii}&=&-\frac{1}{u}[h_{ii1}\langle X,e_1\rangle+h_{ii}-h_{ii}^2u]+[(\gamma')^2+\gamma''](|X|^2_i)^2+\gamma'|X|^2_{ii}.\nonumber
 \end{eqnarray}
Thus,
 \begin{eqnarray}\label{4.4}
 0&\geq & \sigma_k^{ii}\phi_{ii}=-\frac{\langle X,e_1\rangle}{u}\sigma_k^{ii}h_{ii1}-\frac{\sigma_k^{ii}h_{ii}}{u}+\sigma_k^{ii}h_{ii}^2
 +4[(\gamma')^2+\gamma'']\langle X,e_1\rangle^2\sigma_k^{11}\\
 &&\ \  \  \ +\gamma'\sigma_k^{ii}[2-2uh_{ii}]\nonumber.
\end{eqnarray}
By (\ref{6}),
$$
\sigma_k^{ii}h_{ii1}\ \ = \ \  d_X f(e_1)+h_{11}d_{\nu}f(e_1).
$$
Using \eqref{4.3} and $\langle X,e_1\rangle \neq 0$,  we have $$h_{11}=2\gamma'u.$$
Hence, \eqref{4.4} becomes,
\begin{eqnarray}\label{4.5}
 0&\geq &-\frac{1}{u}[\langle X,e_1\rangle d_Xf(e_1)+kf]+\sigma_k^{ii}h_{ii}^2
 +4[(\gamma')^2+\gamma'']\langle X,e_1\rangle^2\sigma_k^{11}\\
 &&+\gamma'\sigma_k^{ii}[2-2uh_{ii}]-2\gamma'\langle X,e_1\rangle d_{\nu}f(e_1) \nonumber.
\end{eqnarray}
Condition (\ref{4.2}) yields,
$$0\geq  \rho^{k-1}[kf+\rho\frac{\p f (X,\nu)}{\p \rho}]=\rho^{k-1}[kf+\rho d_Xf(\frac{\p X}{\p \rho})]=\rho^{k-1}[kf+d_Xf(X)].$$
Since in the local frame, $\langle X,e_i\rangle =0$, for $i\geq 2$, so $X=\langle X, e_1\rangle e_1$. \eqref{4.5} becomes,
 \begin{eqnarray}\label{4.6}
 0&\geq &\sigma_k^{ii}h_{ii}^2
 +4[(\gamma')^2+\gamma'']\langle X,e_1\rangle^2\sigma_k^{11}
+\gamma'\sigma_k^{ii}[2-2uh_{ii}]-2\gamma' d_{\nu}f(X) .
\end{eqnarray}
Choose
$$\gamma(t)=\frac{\alpha}{t},$$ for sufficient large $\alpha$. Therefore,
$$4[(\gamma')^2+\gamma'']|X|^2\sigma_k^{11}+2\gamma'\sum_i\sigma_k^{ii}+\sigma_k^{ii}h_{ii}^2\geq C\alpha^2 \sigma_k^{11},$$ and $$\sigma_k^{11}\geq \sigma_{k-1}\geq \sigma_k^{\frac{k-1}{k}}=f^{\frac{k-1}{k}}.$$
\eqref{4.6} is simplified to
\begin{eqnarray}\label{4.7}
 0&\geq &C_0\alpha^2f^{\frac{k-1}{k}}+\frac{2\alpha}{|X|^4} d_{\nu}f(X) .
\end{eqnarray}
By the assumption on $C^0$ bound, we have  $|d_{\nu}f(X)|\leq C$. Rewrite  \eqref{4.7},
$$0\geq f^{\frac{k-1}{k}}\alpha(C_0\alpha+\frac{2}{k|X|^4}d_{\nu}f^{\frac{1}{k}})>0,$$ for sufficient large $\alpha$, contradiction.  That is, at $X_0$,  $X$ is parallel to the normal direction.  Since $u$ is the support function, $u=\langle X,\nu\rangle=|X|$.
\end{proof}

\begin{theo}\label{them 14}
Suppose $k=2$, and $f$ satisfies condition (\ref{4.1}) and (\ref{4.2}), equation \eqref{1.1}  has only one admissible solution in $\{r_1<|X|<r_2\}$.
\end{theo}
\begin{proof}
We use continuity method to solve the existence result.  For $0\leq t\leq 1$, according to \cite{CNS5}, we consider the family of functions,
\begin{eqnarray}
f^t(X,\nu)=tf(X,\nu)+(1-t)C_n^2[\frac{1}{|X|^k}+\varepsilon(\frac{1}{|X|^k}-1)],\nonumber
\end{eqnarray}
where $\varepsilon$ is sufficient small constant satisfying $$0< f_0\leq \min_{r_1\leq \rho\leq r_2} (\frac{1}{\rho^k}+\varepsilon(\frac{1}{\rho^k}-1)),$$ and $f_0$ is some positive constant.

At $t=0$, we let $X_0(x)=x$. It satisfies $\sigma_2(\kappa(X_0))=C^2_n$. It is obvious that $f^t(X,\nu)$ satisfies the barrier condition in the Introduction (1) and (2) with strict inequality for $0\leq t<1$. Suppose that $X_t$ is the solution of $f^t$. Then, at the maximum point of $\rho_t=|X_t|$, the outer normal direction $\nu_t$ is parallel to the position vector $X_t$.  If that point touches the sphere $|X|=r_2$,
then , at that point, $$\frac{C^2_n}{r^2_2}\leq \sigma_2(\kappa(X_t))=f(X_t,\frac{X_t}{|X_t|})< \frac{C_n^2}{r_2^2}.$$ It is a contradiction. That is $\rho_t\leq r_2$. Similar argument yields that $\rho_t\geq r_1$. $C^0$ estimate follows.

Since the outer normal direction $$\nu=\frac{\rho x-\nabla \rho}{\sqrt{\rho^2+|\nabla \rho|^2}},$$  replace $\rho$ by $t\rho$, $\nu$ does not change. The same argument in \cite{CNS5} gives the openness for $0\leq t <1$.

In view of Evans-Krylov theory, we only need gradient and $C^2$ estimate to complete the closedness part.  With the positive upper and lower bound for $\rho$, Lemma \ref{lemma 12} gives the gradient estimate.  The $C^2$ estimate follows from Theorem \ref{theo2 11}.

The proof of the uniqueness is same as in \cite{CNS5}.
\end{proof}

Now we consider the existence of convex solutions of equation (\ref{1.1}) for the general $k$.
\begin{lemm}\label{lemm 16}
For any strictly convex solution of equation \eqref{1.1} and $f\in C^2(\Gamma\times \mathbb S^n)$,  if $\rho$ have a upper bound,
then the global $C^2$ estimate (\ref{Mc2}) holds.
\end{lemm}
\begin{proof}
First, we will prove that each convex hypersurface satisfying equation \eqref{1.1} contains some small ball  whose radius has a uniform positive lower bound.  Since our hypersurface is convex with an upper bound, we only need to prove that the volume of the domain enclosed by $M$ has a uniform lower positive bound. Let $u$ be the support function of the hypersurface $M$. Since $M$ is strictly convex, the support function $u$ can be viewed a function on the unit sphere.  Let,
$$V_k(M)=\int_{\mathbb{S}^n}\sigma_k(W_u).$$ Here we denote $(W_u)_{ij}=u_{ij}+u\delta_{ij}$. We can rewrite equation \eqref{1.1},
$$\sigma_{n-k}(W_u)=f\sigma_n(W_u)\leq C\sigma_{n}(W_u).$$  Hence,
$$\int_{\mathbb S^n}u\sigma_{n-k}(W_u)\leq C\int_{\mathbb S^n}u\sigma_n(W_u).$$ Therefore,
$$V_{n-k+1}(M)\leq CV_{n+1}(M).$$
Here $V_{n+1}$ is the volume of the domain enclosed by the hypersurface $M$.  By the isoperimetric type inequality of Alexsandrov-Frenchel,
$$V_{n+1}^{\frac{n-k+1}{n+1}}(M)\leq C V_{n-k+1}(M)\leq C V_{n+1}(M).$$ That is, the volume is bounded from below.

For any hypersurface $M$ satisfying \eqref{1.1}, we may assume that the center of the above unit ball is $X_M$.  Let $X-X_M=\rho' y$, where $y$ is another position vector of unit sphere. Obviously, $\rho'$ has positive upper and lower bound. We can view $M$ as a radial graph over the unit sphere centered at $X_M$. By the convexity assumption, $\nabla \rho'$ is bounded by
$\max_{\mathbb S^n} \rho'$. This gives the $C^1$ bound for $M$. Theorem \ref{theo 11} yields global $C^2$ estimate of $\rho'$. Thus, $C^2$ estimate of $\rho$ follows.
\qed

\medskip

\noindent{\bf Proof of Theorem \ref{k-exist}}.
The existence can be deduced by the degree theory as in \cite{gg}. Since the main arguments are the same, we only give an outline.  Consider an auxiliary equation,
\begin{eqnarray}\label{4.10}
\sigma_k(\kappa(X))=f^t(X,\nu),
\end{eqnarray}
where $$f^t=\big(tf^{\frac1k}(X,\nu)+(1-t)(C^k_n[\frac{1}{|X|^k}+
\varepsilon(\frac{1}{|X|^k}-1)])^{\frac1k}\big)^k.$$

By the assumptions in Theorem \ref{k-exist}, $f^t$ satisfies the structural condition in the Constant Rank Theorem (Theorem 1.2 in \cite{GLM1}). This implies the convexity of solutions to equation (\ref{4.10}). Lemma \ref{lemm 16} gives $C^2$ estimates. The Evans-Krylov Theorem yields a priori $C^{3,\alpha}$ estimates. To establish the existence, we only need to compute the degree at $t=0$.
It is obvious that, in this case, $\rho\equiv1$ is the solution.
Then the same computation in \cite{gg} yields the degree in non-zero.
Hence, we have the existence part of the theorem. The strictly convex follows from constant rank theorem in \cite{GLM1}.
\end{proof}

\section{Some examples}

Curvature estimate (\ref{Mc2}) is special for equation (\ref{1.1}). It fails for convex hypersurfaces in $\mathbb R^{n+1}$ for another type of fully nonlinear elliptic curvature equations. We construct such examples for hypersurfaces satisfying the quotient of curvature equation,
\begin{eqnarray}\label{ex}
\frac{\sigma_k(\kappa)}{\sigma_{l}(\kappa)}=f(X, \nu).
\end{eqnarray}

Choose a smooth function $u$ defined on sphere such that the spherical Hessian $$W_u=(u_{ij}+u\delta_{ij} )\in \Gamma_{n-1}$$ but
$\sigma_n(W_u(y_0))<0$ at some point $y_0\in \mathbb S^n$. The existence of such functions are well known (e.g., \cite{A}). Set $\tilde f=\sigma_{n-1}(W_u)$, so $f$ is a positive and smooth function. Set
$$u_t=(1-t)+tu.$$
We have $W_{u_t}\in \Gamma_{n-1}$ and
\begin{equation}\label{ft} \tilde f_t=\sigma_{n-1}(W_{u_t}),\end{equation} is smooth and positive. Obviously, when t is close to $0$, $W_{u_t}$ is positive definite. There is some $1>t_0>0$, such that $W_{u_t}>0$ for $t<t_0$, and   $$\det(W_{u_{t_0}}(x_0))=0,$$
for some $x_0\in \mathbb{S}^n$. Denote $\Omega_u$ to the convex body determined by its support function $u_t$, $0\le t<t_0$.

\noindent
{\bf Claim:}  {\it for each $0\le t<t_0$
after a proper translation of the origin, we have some positive constant $c_0$ independent of $t<t_0$ such that,
\begin{eqnarray}\label{utc0}
u_t(x)\geq c_0>0\ \ \text{ for }\forall x\in\mathbb{S}^n\ \ \text{ and } t<t_0. \end{eqnarray}
That is each $\Omega_{u_t}$ contains a ball of fixed radius, $t<t_0$. }

Let's first consider $k=n, l=1$ in equation (\ref{ex}). For $0\le t<t_0$, denote
\begin{equation}\label{mt}
M_t=\partial \Omega_{u_t}. \end{equation}
For each $0\le t<t_0$, $M_t$ is strictly convex. By (\ref{utc0}), we have uniform $C^1$ estimate for the radial function $\rho_t$, where $M_t=\{\rho_t(z) z | z\in \mathbb S^n\}$. We can rewrite the equation \eqref{ft},
\begin{equation}\label{ftt}\frac{\sigma_n}{\sigma_{1}}(\kappa_1,\cdots, \kappa_n)=\frac{1}{\tilde f_t(\nu)}.\end{equation} Since $\sigma_n(W_{u_{t_0}}(x_0))=0$, the principal curvature of $M_t$ will blow up at some points as $t\to t_0$. The uniform curvature estimate (\ref{Mc2}) for equation \eqref{ftt} can not hold.

We prove {\bf claim}. Fix $0\le t<t_0$, after a proper translation, we may assume the origin is inside the convex body $\Omega_{u_t}$. It follows from the construction,
$$\tilde f_t\geq c>0,$$ for some constant $c>0$ and for any $t<t_0,x\in \mathbb{S}^n$, and
\begin{eqnarray}\label{paper}
\|u_t\|_{C^3(\mathbb{S}^n)}\leq C,
\end{eqnarray}
where constant $C$ is independent of $t$. At the maximum value points $x_0^t$ of functions $u_t$, we have,
$$W_{u_t}(x_0^t)\leq u_{t}(x_0^t)I.$$ Hence, $$u_t(x_0^t)\geq \tilde f_t^{\frac{1}{n-1}}(x_0^t)\geq C>0.$$  Estimate (\ref{paper}) implies that there is some uniform radius
$R$ such that on the disc $B_{R}(x_0^t)$ with center at $x_0$,
$$u_t(x)\geq \frac{C}{2}>0, \forall x\in B_{R}(x_0^t).$$  By the Minkowski  identity,
$$\int_{\mathbb{S}^n}\sigma_n(W_{u_t})=c_n\int_{\mathbb{S}^n}u_t\sigma_{n-1}(W_{u_t})=c_n\int_{\mathbb{S}^n}u_tf_t\geq c_n\int_{B_R(x_0^t)\cap\mathbb{S}^n}u_t\tilde f_t\geq \tilde{c}>0.$$ Hence, there exists $y_0^t\in\mathbb{S}^n$ satisfying $$\sigma_n(W_{u_t}(y_0^t))\geq \frac{\tilde{c}}{\omega_n}.$$
By \eqref{paper}, there are some uniform radius $\tilde{R}>0$, such that for $y\in \mathbb{S}^n\cap B_{\tilde{R}}(y_0^t)$, we have,
$$W_{u_t}(y)\geq \frac{\tilde{c}}{2\omega_n} >0.$$ Hence, near the points $\nu^{-1}(y_0^t)$, the hypersuface $M_t$ is pinched by two fixed paraboloids locally and uniformly. Thus,
$\Omega_{u_t}$ contains a small ball whose radius has a uniform positive lower bound. Move the origin to the center of the ball, this yields (\ref{utc0}). The claim is verified.

\medskip

\noindent
{\it Proof of Theorem \ref{theo 11-counter}.} We use the some sequence $\{M_t\}$ defined in (\ref{mt}) to construct $f_t$ in (\ref{cur-qq}). For any $m=0, 1, \cdots, n-1$, for any $0\le t<t_0$, $\sigma_m(W_{u_t})\in C^{\infty}(\mathbb S^n)$. By (\ref{ft}), (\ref{paper}) and Newton-MacLaurin inequality, there exists $c>0$ independent of $t$,
\[ c \le \sigma_m(W_{u_t})\le \frac1c.\]
Since $\frac{\sigma_{k}}{\sigma_l}(\kappa_{M_t}) \equiv \frac{\sigma_{n-k}}{\sigma_{n-l}}(W_{u_t})$, there exists $a>0$ independent of $t$, such that for any $1\le l<k\le n$,
\[a\le \frac{\sigma_{k}}{\sigma_l}(\kappa_{M_t}) \le \frac1a.\]
$M_t$ satisfies equation
\[\frac{\sigma_{k}}{\sigma_l}(\kappa_{M_t})=\frac{\sigma_{n-k}}{\sigma_{n-l}}(W_{u_t})=f_t(\nu),\]
$f_t, \frac1{f_t}\in C^{\infty}(\mathbb S^n)$ and the norms of $\|f_t\|_{C^3(\mathbb S^n)}$ and
$\|\frac1{f_t}\|_{C^3(\mathbb S^n)}$ under control independent of $0\le t<t_0$. That is, $M_t$ satisfies conditions in theorem. The previous analysis on $M_t$ indicates that estimate (\ref{Mc2}) fails and the principal curvature of $M_t$ will blow up at some points when $t\to t_0$. \qed

\bigskip

\noindent {\it Acknowledgement:} The work was done while the second author was visiting McGill University. He would like to thank the China  Scholarship Foundation for their support. He would also like to thank McGill University for their hospitality.


\begin{thebibliography}{99}

\bibitem{A}
A.D. Alexandrov,{\em Zur Theorie der gemischten Volumina von konvexen Körpern, III.
Die Erweiterung zweier Lehrsätze Minkowskis über die konvexen Polyeder auf beliebige
konvexe Flächen}, (in Russian). Mat. Sb. \textbf{3}, (1938), 27-46.
\bibitem{A1} A.D. Alexandrov,
{\em Existence and uniqueness of a convex
surface with a given integral curvature}, Doklady Acad. Nauk Kasah SSSR
\textbf{36} (1942), 131-134.

\bibitem{A2}
A.D. Alexandrov, {\em Uniqueness theorems for surfaces in the
large. I} (Russian), Vestnik Leningrad. Univ., \textbf{11},
(1956), 5-17. English translation: AMS Translations, series 2,
\textbf{21}, (1962), 341-354.

\bibitem {BK}
I. Bakelman and B. Kantor,
{\em Existence of spherically homeomorphic hypersurfaces in Euclidean
space with prescribed mean curvature},
 Geometry and Topology, Leningrad,
\textbf{1}, (1974), 3-10.

\bibitem{Ball} J. Ball, {\em Differentiability properties of symmetric and isotropic functions}. Duke Math. J. \textbf{51}, (1984),  699-728.

\bibitem{CNS1} L. Caffarelli,  L. Nirenberg and J. Spruck,
{\em The Dirichlet problem for nonlinear  second order elliptic equations
I. Monge-Amp\`ereequations},
Comm. Pure Appl. Math. \textbf{37}, (1984), 369-402.

\bibitem{CNS3}
L. Caffarelli, L. Nirenberg and J. Spruck,
{\em The Dirichlet problem for nonlinear second order elliptic equations,  III:
Functions of the eigenvalues of the Hessian}
Acta Math. \textbf{155}, (1985),261 - 301.

\bibitem{CNS5}
L. Caffarelli, L. Nirenberg and J. Spruck,
{\em Nonlinear second order elliptic equations IV:
\quad Starshaped compact Weigarten hypersurfaces},
 Current topics in partial differential equations, Y.Ohya, K.Kasahara
and N.Shimakura (eds),
 Kinokunize, Tokyo, 1985,
 1-26.

\bibitem{CY}
S.Y. Cheng and S.T. Yau, {\em  On the regularity of the solution
of the n-dimensional Minkowski problem}, Comm. Pure Appl. Math.,
\textbf{29}, (1976), 495-516.

\bibitem{CW} K.S. Chou and X.J. Wang, {\em A variational theory of the Hessian equation}. Comm. Pure Appl. Math. \textbf{54}, (2001), 1029-1064.

\bibitem{Gerh} C. Gerhardt, {\em Hypersurfaces of prescribed scalar curvature in Lorentzian manifolds}. J. Reine Angew. Math. \textbf{554}, (2003), 157-199.

\bibitem{B} B. Guan, {\em The Dirichlet problem for Hessian equations on
Riemannian manifolds }, Calc. Var. \textbf{8}, (1999), 45-69.

\bibitem{gg} B. Guan and P. Guan, {\em Convex Hypersurfaces of Prescribed Curvature},
Annals of Mathematics, \textbf{156},
(2002), 655-674.


\bibitem{GLL}P. Guan, J. Li, and Y.Y. Li, {\em Hypersurfaces of Prescribed Curvature Measure}, Duke Math. J. \textbf{161}, (2012), 1927-1942.

\bibitem{GLM} P. Guan, C.S. Lin and X. Ma, {\em The Existence of Convex Body with Prescribed Curvature Measures},
International Mathematics Research Notices, \textbf{2009}, (2009) 1947-1975.

\bibitem{GLM1}
P. Guan, C.S Lin and X. Ma{\em The Christoffel-Minkowski problem II: Weingarten curvature equations},
Chin. Ann. Math., \textbf{27B}, (2006), 595-614.

\bibitem{GM}
P. Guan and X. Ma, {\em Christoffel-Minkowski problem I: convexity
of solutions of a Hessian equation}, Invent. Math., \textbf{151},
(2003), 553-577.

\bibitem{HS} Gerhard Huisken and Carlo Sinestrari, Convexity estimates for mean curvature flow and singularities of mean convex surfaces.
Acta. Math., \textbf{183}, (1999), 45-70.

\bibitem{I1}
N. Ivochkina, {\em Solution of the Dirichlet problem for curvature
equations of order m}, Mathematics of the USSR-Sbornik, \textbf{67}, (1990),
 317-339.

\bibitem{I}
 N. Ivochkina. {\em The Dirichlet problem for the equations of curvature of
order m. Leningrad Math. J.} \textbf{2-3}, (1991), 192-217.

\bibitem{N}
L. Nirenberg, {\em The Weyl and Minkowski problems in differential
geometry in the large}, Comm. Pure Appl. Math. \textbf{6}, (1953),
337-394.

\bibitem{P1}
A.V. Pogorelov, {\em On the question of the existence of a convex
surface with a given sum principal radii of curvature ( in
Russian)},  Uspekhi Mat. Nauk \textbf{8}, (1953), 127-130.

\bibitem{P3} A.V. Pogorelov, {\em The Minkowski Multidimensional
Problem}, John Wiley, 1978.

\bibitem{TW}
 A. Treibergs and S.W. Wei,
{\em Embedded hypersurfaces with prescribed mean curvature},
 J. Diff. Geometry,
\textbf{18}, (1983),  513-521.

\end{thebibliography}
\end{document}